\documentclass{article}
\usepackage[utf8]{inputenc}
\usepackage[full]{textcomp}
\usepackage[osf]{newtxtext}
\usepackage{comment}

\usepackage{amssymb}
\usepackage{mathtools}
\usepackage{hyperref}
\usepackage{breakurl}
\usepackage{mhenvs}
\usepackage{mhequ}
\usepackage{mhsymb}
\usepackage{booktabs}
\usepackage{tikz}
\usepackage{mathrsfs}
\usepackage{longtable}
\usepackage{times}

\newtheorem*{lemma**}{Lemma}
\newtheorem*{theorem**}{Theorem}

\usepackage{microtype}
\usepackage{wasysym}
\usepackage{centernot}
\usepackage{enumitem}

\numberwithin{equation}{section}

\makeatletter
\newcommand{\globalcolor}[1]{%
  \color{#1}\global\let\default@color\current@color
}
\makeatother

\newif\ifdark
\IfFileExists{dark}{\darktrue}{\darkfalse}

\ifdark

\definecolor{darkred}{rgb}{0.9,0.2,0.2}
\definecolor{darkblue}{rgb}{0.7,0.3,1}
\definecolor{darkgreen}{rgb}{0.1,0.9,0.1}
\definecolor{pagebackground}{rgb}{.15,.21,.18}
\definecolor{pageforeground}{rgb}{.84,.84,.85}
\pagecolor{pagebackground}
\AtBeginDocument{\globalcolor{pageforeground}}

\else

\definecolor{darkred}{rgb}{0.7,0.1,0.1}
\definecolor{darkblue}{rgb}{0.4,0.1,0.8}
\definecolor{darkgreen}{rgb}{0.1,0.7,0.1}
\definecolor{pagebackground}{rgb}{1,1,1}
\definecolor{pageforeground}{rgb}{0,0,0}

\fi

\theoremnumbering{Alph}
\renewtheorem{theorem*}{Theorem}

\DeclareMathAlphabet{\mathbbm}{U}{bbm}{m}{n}
\DeclareFontFamily{U}{BOONDOX-calo}{\skewchar\font=45 }
\DeclareFontShape{U}{BOONDOX-calo}{m}{n}{
  <-> s*[1.05] BOONDOX-r-calo}{}
\DeclareFontShape{U}{BOONDOX-calo}{b}{n}{
  <-> s*[1.05] BOONDOX-b-calo}{}
\DeclareMathAlphabet{\mcb}{U}{BOONDOX-calo}{m}{n}
\SetMathAlphabet{\mcb}{bold}{U}{BOONDOX-calo}{b}{n}

\let\epsilon\varepsilon
\def\E{{\symb E}}
\def\F{{\mathcal F}}

\def\FC{\mathscr{C}}
\def\L{\mathbb L}

\def\X{\mathbb{X}}

\def\WW{{ \mathbb W}}
\def\X{{\mathbf X}}
\def\XX{{\mathbb X}}

\def\Y{{\mathbf Y}}
\def\YY{{\mathbb Y}}

\def\D{{\mathcal D}}
\def\err{\mathbf {Er}}

\def\C{\mathcal{C}}
\def\f{\frac}
\def\1{\mathbf{1}}

\def\cov{{\mathrm{Cov}}}
\def\${|\!|\!|}

\def\id{\mathrm{id}}

\def\<{\langle}
\def\>{\rangle}
\setlist{noitemsep,topsep=4pt}
\def\para_#1{/\!\!/_{\!#1}}

\def\slash{\kern0.18em/\penalty\exhyphenpenalty\kern0.18em}
\def\dash{\kern0.18em--\penalty\exhyphenpenalty\kern0.18em}

\makeatletter 
\newcommand*{\fat}{}
\DeclareRobustCommand*{\fat}{%
\mathbin{\mathpalette\bigcdot@{}}}
\newcommand*{\bigcdot@scalefactor}{.5}
\newcommand*{\bigcdot@widthfactor}{1.15}
\newcommand*{\bigcdot@}[2]{%
  \sbox0{$#1\vcenter{}$}
  \sbox2{$#1\cdot\m@th$}%
  \hbox to \bigcdot@widthfactor\wd2{%
    \hfil
    \raise\ht0\hbox{%
      \scalebox{\bigcdot@scalefactor}{%
        \lower\ht0\hbox{$#1\bullet\m@th$}%
      }%
    }%
    \hfil
  }%
}
\makeatother
{\theorembodyfont{\rmfamily}

\newtheorem{convention}[lemma]{Convention}
}

\newtheorem{assumption}[lemma]{Assumption}

\begin{document}
\title{Diffusive and rough homogenisation in fractional noise field}

\author{Johann Gehringer and Xue-Mei~Li\\
Imperial College London 
\footnote{johann.gehringer18@imperial.ac.uk, xue-mei.li@imperial.ac.uk}
}

\maketitle

\begin{abstract}
	With recently developed tools, we prove a homogenisation theorem for a random ODE with short and long-range dependent fractional noise.
	The effective dynamics are not necessarily diffusions, they are given by stochastic differential equations driven simultaneously by stochastic processes from both the Gaussian and the non-Gaussian self-similarity universality classes. A key lemma for this is the `lifted' joint functional central and non-central limit theorem in the rough path topology.
\end{abstract}

{ \scriptsize {\it  keywords:} passive tracer, fractional noise, multi-scale,  functional limit theorems, rough differential equations}


\setcounter{tocdepth}{2}
\tableofcontents
\section{Introduction}

Fractional noise is the `derivative' of a fractional Brownian motion. Its covariance at times separated by a span $s$ is
$\tilde \rho(s)\sim 2H(2H-1) |s|^{2H-2} +2H |s|^{2H-1}\delta_{s}$ where $H$ is the Hurst parameter taking values in $(0,1)\setminus\{\f 12\}$ and $\delta_s$ is the Dirac measure. The `$H=\f 12$' case is white noise.  If $H>\f 12$, $\int_\R \tilde \rho ds =\infty$ which means that the noise has non-integrable long range dependence (LRD).    If $H<\f12$, the process is negatively correlated. 
Just as white noise is used for modelling noise coming from a large number of independent random components,  fractional noise is used for modelling Long range dependence (LRD). LRDs are observed in nature and in time series data. We study the two scale passive tracer problem, this is also called the tagged particle problem,  with fractional noise.

We consider a slow/fast system  in which the slow variables are given by a random ODE $\dot x_t=G(x_t, y_t^\epsilon)$.
This touches on  two problems.  The first is the passive tracer problem modelling the motion of a tagged particle in a disturbed flow, not necessarily incompressible, which allows simulation of the turbulent from the Lagrangian description. The other is the dynamical description for Brownian particles in a liquid at rest. The slow variables evolve in their natural time scale, while the fast random environment evolves in the microscopic scale $\epsilon$.
  The aim is to extract a closed effective dynamics which approximates the slow variables when $\epsilon$ is sufficiently small.  This effective dynamics will be obtained from the persistent effects coming from the fast-moving variables through adiabatic transmission. If the environment is stationary strong mixing noise with sufficiently fast rate of convergence, the homogenisation problem is synonymous with `diffusion creation', and is therefore also known as diffusive homogenisation. There have been continuous explorations of the diffusive homogenisation problem, see  \cite{Green, Hasminskii, Kubo,  Kipnis-Varadhan, Landim-Olla-Varadhan, Papanicolaou-Kohler,Taylor, Komorowski-Landim-Olla} and the references therein. 
Recently long range dependent noises are also studied in several papers in the context of homogeneous incompressible fluids, however, they inevitably fall within the central limit theorem regimes \cite{Fannjiang-Komorowski-2000,Komorowski-Novikov-Ryzhik-12} and the effective dynamics are either Brownian motions or fractional Brownian motions.

 We will study a family of vector fields without spatial homogeneity, the resulting dynamics can take the form of a process resembles locally a fractional Brownian motion and more generally they compromise of a larger class of stochastic dynamical systems of the form
\begin{equation}\label{limit-eq}
 d{x}_t =\sum_{k=1}^n f_k(x_t) \circ d X^k_t+\sum_{k=n+1}^N f_k(x_t) d X^k_t, \quad x_0=x_0,
\end{equation}
where $X^k_t$ is a Wiener process for $k \leq n$ and otherwise a Gaussian or a non-Gaussian Hermite process. To our best knowledge, this presents a new effective limit class. In these equations, the symbol $\circ$ denotes the Stratonovich integral and the other integrals are in the sense of Young integrals.  
  
The homogenisation problem we consider is: 
	\begin{equation}\label{multi-scale}
\left\{ \begin{aligned} \dot x_t ^\epsilon &=\sum_{k=1}^N\alpha_k(\epsilon) \, f_k(x_t^\epsilon) \,G_k(y_t^\epsilon),\\
x_0^\epsilon&=x_0, \end{aligned}\right.  
\end{equation}
where $y^\epsilon=y_{\f t \epsilon}$ and $y_t$ are the short and long range dependent stationary fractional Ornstein-Uhlenbeck processes (fOU) with Hurst parameter $H  \in (0,1) \setminus \{ \f 1 2 \} $ and one time probability distribution $\mu$, the centred real valued functions $G_k\in L^p(\mu)$ transforms the noise. 
If $f_k$ are in $\C_b^1$ and $G_k$ are bounded measurable,  the solutions to the equations $\dot x_t^\epsilon=  \sum_{k=1}^N f_k( x_t^\epsilon) G_k( y_t^\epsilon)$ will be approximated by the averaged dynamics which, in this case, is the trivial ODE $\dot x_t=0$, c.f. \cite{Hairer-Li} and \cite{Li-Sieber}.  A homogenisation theorem will then describe the fluctuation around this average, for this we must rescale the vector fields to arrive at a non-trivial limit. The different scales $\alpha_k(\epsilon)$ are reflections of the non-strong mixing property of the  noise, they tend to $\infty$ as $\epsilon\to 0$ at a speed tailored to the transformations~$G_k$.  These scales determine the local self-similar property of the limit. If $G$ is an $L^2$ function with Hermite rank $m$, to be defined below,  then $m=\f 1{2(1-H)} $  is the critical value  for the limit to be locally a Brownian motion. If $m$ is smaller,  the effective limit is locally a Hermite process of rank $m$, otherwise a Wiener process.

Our main theorem is the following.  We take $\alpha_k(\epsilon)$  to be $\alpha\left(\epsilon,H^*(m_k)\right)$, the latter is defined by (\ref{alpha}). 

\begin{theorem*}\label{main theorem}
Let $H \in (0,1)\setminus \{ \f 1 2 \}$, $f_k\in \C_b^3( \R^d ;\R^d)$ and $G_k\in L^{p_k}(\R;\R, \mu)$ be real valued functions satisfying Assumption~\ref{assumption-multi-scale}.
Then the solutions of (\ref{multi-scale}) converge weakly  in  $\C^\gamma$, on any finite time interval and for any $\gamma \in (\f 1 3 ,\f 1 2 - \f 1 {\min_{k \leq n } p_k})$, to the solution of  (\ref{limit-eq}).
\end{theorem*}

The linear contraction in the Langevin equation and the exponential convergence of the solutions would lead to the belief that it mixes as fast as the Ornstein-Uhlenbeck process.  But, the auto-correlation functions of the increment process, which measures how much the shifted process remembers, exhibits power law decay.    For $H>\f 12$, the auto correlation function is not integrable. 
Conventional tools are not applicable here, we turn to  the theory of
rough path differential equations and view (\ref{multi-scale}) as rough differential equations driven by
stochastic processes with a parameter $\epsilon$. By the continuity theorem for solutions of rough differential equations, it is then sufficient to prove the convergence of these drivers in the rough path topology.
 For continuous processes this concerns the scaling limits of the path integrals of the form $\int_0^t G_k( y^{\epsilon}_s)ds$ together with their canonical lifts. Using rough path theory for stochastic homongenisation is a recent development, in \cite{Kelly-Melbourne, roughflows}, this was used for diffusive homogenisation. Proving and formulating an appropriate functional limit theorem, however, turned out to be one of our main endeavours.

For independent  identically distributed random variables, the central limit theorems (CLTs) states that 
$\f 1 {\sqrt n} \sum_{k=1}^n X_k$ converges to a Gaussian distribution. For correlated random variables, non-Gaussian distributions may appear. One of these was proved by Rosenblatt: 
Let $Z_n$ be a stationary Gaussian sequence with correlation $\rho(n)\sim n^{-d}$ where $d\in (0, \f 12)$ and let $Y_n=(Z_n)^2-1$ then ${n^{d-1}} Y_n$ converges to a non-Gaussian distribution.
 To emphasise the non-Gaussian nature, those limit theorems with non-Gaussian limits are referred to  {\it `non-Central Limit Theorems'} (non-CLTs). 
  A functional limit theorem concerns path integrals of functionals of a stochastic process $y_t$. For a centred function $G$, it states that $\lim_{\epsilon\to 0} \sqrt \epsilon\int_0^{\f t \epsilon}G(y_s) ds$ converges to a Brownian motion.  Non-CLTs and functional non-CLTs were extensively studied \cite{Maejima-Ciprian, BenHariz, Breuer-Major,Taqqu-75},  these were then shown to hold for a larger class of functions   \cite{ Campese-Nourdin-Nualart, Nualart-Peccati} with Malliavin calculus. In a nutshell, for a class of Gaussian processes and for a centred  $L^2$ function $G$ with the scaling constant depending on its Hermite rank $m$, the limit 
of $ \alpha(\epsilon)\int_0^{t} G(y^{\epsilon}_s)ds$  will be a BM if the scale is  $\f 1 {\sqrt{\epsilon}}$  or$\f  1 {\sqrt{  \epsilon \vert \ln\left( \epsilon \right) \vert}}$; otherwise it is a self-similar Hermite process of degree $m$ with self-similar exponent $H^*(m)=m(H-1)+1$. We will use functional limit theorems for both cases. 

Let $\alpha(\epsilon, H^*(m))$ be positive constants as follows, they depend on $m, H$ and $\epsilon$ and tend to $\infty$  as $\epsilon\to 0$,
 \begin{equation}\label{alpha}
\begin{aligned}
&\alpha\left(\epsilon,H^*(m)\right) = \left\{\begin{array}{cl}
\f 1 {\sqrt{\epsilon}}, \, \quad  &\text{ if } \, H^*(m)< \f 1 2,\\
\f  1 {\sqrt{  \epsilon \vert \ln\left( \epsilon \right) \vert}}, \, \quad  &\text{ if } \, H^*(m)= \f 1 2, \\
\epsilon^{H^*(m)-1}, \quad  \, &\text{ if } \,  H^*(m) > \f 1 2.
\end{array}\right.
\end{aligned}
\end{equation}

Observe that $H^*$ decreases with $m$ and $H^*(1)=H$. If $H\le\f 12$ we only see the diffusion scale.
We state below our key limit theorem, the lifted joint functional limit theorem in the rough path topology, c.f.  (\ref{rough-distance}), see \S \ref{iterated-CLT}. The proof for the main theorem is finalised in \S\ref{conclusion}.

\begin{theorem*}[{\bf Lifted joint functional CLTs/ Non-CLTs}] \label{theorem-lifted-CLT} 
	Let $H \in (0,1)\setminus \{ \f 1 2 \}$ and fix a finite time horizon $T$. Suppose that the $L^2(\mu)$ functions $G_1, \dots, G_N $  satisfy Assumption \ref{assumption-multi-scale}. Let $m_k$ denote the Hermite rank of $G_k$.  Set 
	\begin{equation}
	X^{k,\epsilon}_t = \alpha(\epsilon,H^*(m_k))\int_0^{t} G_k(y^{\epsilon}_s)ds, \qquad X^{\epsilon} = (X^{1,\epsilon}_t , X^{2,\epsilon}_t, \dots, X^{N,\epsilon}_t).
	\end{equation}
	 \begin{enumerate}
			\item  Then, for every $\gamma \in (\f 1 3 , \f 1 2 - \f {1} {\min_{k \leq n } p_k})$,
			the canonical rough paths $\X^{\epsilon} := \left(X^\epsilon_t,
			 \XX^{\epsilon}_{s,t}  \right)$ converge weakly in the rough topology $\FC^\gamma([0,T],\R^N)$ and
			 $$\lim_{\epsilon\to 0} \X^{\epsilon} = \X:= \left(X_t, \XX_{s,t}+(t-s)A\right) $$ 
			
			\item The precise formulation for the stochastic process $X_t$ in the limit is given in Theorem \ref{theorem-CLT}. It consists of two independent blocks: a Wiener process block and a Hermite process block.
	 For $0 \leq s\leq t \leq T$, the limiting second order processes are given by $\X=(\XX^{i,j})$ and $A=(A^{i,j})$ where
			$$
			\XX^{i,j}_{s,t}=\int_s^t  (X_r^i - X_s^i) d X_r^j, 
			\qquad  \left\{
			\begin{aligned}
			& \hbox{an It\^o integral},  \qquad & \hbox{ for } i, j\le n,\\
			&\hbox{a Young integral}, & \hbox{otherwise.}\end{aligned}\right.$$
			
			$$A^{i,j}=\left\{	\begin{aligned}
			& \int_0^{\infty} \E\left( G_i(y_s) G_j(y_0) \right) ds, \qquad & \hbox{ if } i, j\le n,\\
			&0, & \hbox{otherwise.}
			\end{aligned}\right.			\hskip 110pt	$$
		\end{enumerate}
\end{theorem*}

The Hermite processes  in Theorem B are $Z^{H^*(m_k),m_k}_t $, see \S\ref{sec:Hermite}.
They have H\"older continuous sample paths up to the order $H^*(m_k)$.
For this theorem, we use a basic functional CLT from \cite{Gehringer-Li-fOU} for  proving the joint convergence of the integrals and their iterated integrals in an appropriate path space,  in finite dimensional distribution. For the Wiener limit part, we employ both ergodic theorems and martingale approximations.  In case where the processes are not strong mixing, proving the $L^2$ boundedness of the martingale approximations is rather involved (this is where we had to exclude functions with Hermit rank falling into the range $[\f 1 {2(1-H)}, \f 1 {1-H}$). We will follow an idea in  \cite{Hairer05, Hairer-Li} for fractional Brownian motions to develop a locally independent decomposition for the fOU process and use this for estimating the conditional moments.  The final hurdle is the relatively compactness of the iterated integrals in the rough path topology,  for which we use the diagram formula and an 
upper bound, from \cite{Graphsnumber}, on the number of eligible graphs of complete pairings.

\bigskip

{\it Acknowledgement.} { 1. We would like to thank M. Gubinelli and M. Hairer for very helpful discussions.}
2.~Previously, we proved the homogenisation theorem for $H>\f 12$. This was posted to the Mathematics arxiv and unpublished otherwise, see  \cite{Gehringer-Li-homo}.  Here we can also include the $H<\f 12$ case. For the presentation,  we did not include the basic joint functional limit theorem from  \cite{Gehringer-Li-homo}. Instead,
an improved version  is presented in  \cite{Gehringer-Li-fOU}.

\subsection*{Notation}

\begin{itemize}
\item $(W_t, t \in \R)$ denotes a two-sided Wiener process.
\item  $B_t$ is the fBM in the Langevin equation, $H$ is its Hurst parameter,
$\F_t$ denotes its filtration.
\item $H^*(m) = m(H-1) +1 $.
\item $m_k$ is the Hermit rank of $G_k$. 
\item Convention :   $H^*(m_k)\le \f12$ for $k\le n$; otherwise  $H^*(m_k)>\f 12$,
\item $\C_b^r$:  bounded continuous functions with bounded continuous derivatives up to order $r$. 
\item $f\lesssim g$ means that there exists a constant $c$, not depending on $f$ or $g$, such that $f\le cg$.
 \item  $|x|_\alpha :=\sup_{s\neq t} \f{| x_t-x_s|}{|t-s|^\alpha}$ is  the homogeneous H\"older semi-norm, $0<\alpha <1$.
  \item For a process $x_t$,  set  $x_{s,t}:=x_t - x_s$.
  \item
  We fix a probability space $(\Omega, \CF,  \P)$.  $L^p(\Omega)$  denotes the  $L^p$ space on $\Omega$ and  its norm is denoted by  $\Vert \fat \Vert_{L^p}$.
 \item  $\mu=N(0,1)$ is the standard Gaussian measure, 
    $L^p(\mu)$ denotes the corresponding $L^p$ space.
\end{itemize}

\section{Preliminaries}\label{preliminary}
A  fractional Brownian motion  is a continuous Gaussian process with stationary increments. 
We take a normalised fractional Brownian motion $B_t$ so that $B_0=0$ and $\E(B_1)^2=1$.   Specifically, if  $H$ is its Hurst parameter, then \begin{equation*}
\E \left( (B_t-B_s)(B_u-B_v) \right)=\f 12 \left( |t-v|^{2H}+|s-u|^{2H}-|t-u|^{2H}-|s-v|^{2H} \right).
\end{equation*}
We refer to  \cite{Pipiras-Taqqu-book, Samorodnitsky, Cheridito-Kawaguchi-Maejima} for details on fractional Brownian motions. Note that $$\E (B_tB_s)= \f 1 2 \left( t^{2H} + s^{2H} - \vert t-s \vert^{2H} \right) =H(2H-1) \int_0^t\int_0^s |r_1-r_2|^{2H-2}dr_1dr_2,$$
and so $\f{ \partial^2}{\partial t\partial s} \E(B_tB_s) =H(2H-1)|t-s|^{2H-2}$,  when $H \in (0,1) \setminus \{ \f 1 2 \} $. Let  $X_n=B_{1+n}-B_n$ denote the increment process of a fBM. Then, the autocorrelation function of $\{X_n\}$ is not summable for $H> \f 1 2$.

\subsection{Hermite processes}\label{sec:Hermite}
Let $W_t$ be a one dimensional standard two-sided Brownian motion. Let  $ \hat H(m)=\f  1 m (H-1) + 1$, so $\hat H$ is the inverse of $H^*$.
\begin{definition}\label{Hermite-processes}
	Let $m\in \N$ with $\hat H(m)>\f 12$. We take a standalised {\it Hermite process} of rank $m$ to be the following mean zero process:
	\begin{equation}\label{Hermite}
	Z_t^{H,m}=\f  {K(H,m)} {m!} \int_{\R^m} \int_0^t \prod_{j=1}^m (s-\xi_j)_+^{  -(\f 1 2 + \f {1-H} {m})} \, ds \,  d  W({\xi_1}) \dots d  W({\xi_m}).
	\end{equation}
The integral  over $\R^m$ is understood as a multiple Wiener-It\^o integral (no integration along the diagonals) and the constant $K(H,m)$ is chosen so that it variance is $1$ at $t=1$. The number $H$ is its self-similarity exponent, it is also known as its Hurst parameter.
\end{definition}
Since $\hat H(1)=H$, the rank $1$  Hermite processes $Z^{H, 1}$ are fractional BMs. Indeed (\ref{Hermite}) is  exactly the 
Mandelbrot Van-Ness representation for a fBM. We emphasise this representation:
$$B_t^H= \int_\R \int_0^t (s-\xi)_+^{ H-\f 32} \, ds \,  d  W_\xi.$$
The Hermite processes have stationary increments, finite moments of all orders and the following covariance function: \begin{equation}
\E(  Z_t^{H,m} Z_s^{H,m}) =  \f 1 2 (   t^{2H} + s^{2H} - \vert t-s \vert^{2H}).
\end{equation}
Therefore, using Kolmogorv's theorem, one can show that the Hermite processes $Z_t^{H,m}$ have sample paths of H\"older regularity up to  $H$. 
As mentioned before, they  also self similar stochastic processes:
$$ \lambda^H  Z^{H,m} _{\f \cdot\lambda} \sim  Z^{H,m}_..$$
The process $Z^{H,m}_t$ belongs to the $m^{th}$ Wiener chaos generated by $W$, in particular,  two  Hermite processes  $Z^{H, m}$ and $Z^{H',m'}$, defined by the same Wiener process, are uncorrelated if $m \not = m'$.
Further details on Hermite processes can also be found in \cite{Maejima-Ciprian}.

\begin{remark}
We note that in  some literature, e.g. \cite{Maejima-Ciprian},
the notation for the Hermite processes are different:
	$$\tilde Z_t^{H,m} =\f  {K(H,m)} {m!} \int_{\R^m} \int_0^t \prod_{j=1}^m (s-\xi_j)_+^{ H- \f 3 2} \, ds \,  d  W({\xi_1}) \dots d  W({\xi_m}).$$
	These two are related by
	\begin{equation}Z_t^{H^*(m), m}=\tilde Z_t^{H,m}, \qquad Z_t^{H,m}=\tilde Z_t^{\hat H(m),m}.
	\end{equation}
\end{remark} 

\subsection{Fractional Ornstein-Uhlenbeck processes}\label{OU-section}
We gather in this section to useful facts about the  stationary fractional Ornstein-Uhlenbeck process, by which we mean 
$y_t =  \sigma \int_{-\infty}^t e^{-(t-s) } dB^H_s$ for  $B^H_t$ a  two-sided fractional BM and $\sigma$  chosen such that  $y_t$ is distributed as $\mu=N(0,1)$.
It is the stationary solution of the Langevin equation: 
$dy_t = - y_t dt + \sigma d B^H_t$ with the initial value $ y_0 = \sigma \int_{-\infty}^0 e^{ s }  dB^H_s$. 
We take rescale the fOU process to obtain $y_t^\epsilon$, the latter is the the  stationary solution of
\begin{equation}\label{fOU}
dy_t^\epsilon = -\f 1 \epsilon y_t^\epsilon\, dt + \f { \sigma} {{\epsilon}^H}\, d B^H_t.
\end{equation}
Observe that $ y_\cdot ^\epsilon$ and $ y_{\f \cdot \epsilon} $ have the same distributions, furthermore, 
$
y^\epsilon_t=\f \sigma {\epsilon^H}\int_{-\infty}^t e^{-\f 1 \epsilon (t-s) } dB^H_s$.
Let  us denote their correlation functions by $\rho$ and $\rho^\epsilon$ respectively:
$$\rho(s,t):=\E(y_sy_t), \qquad \rho^{\epsilon}(s,t):= \E(y^{\epsilon}_sy^{\epsilon}_t).$$ 
Let $\rho(s)=\E (y_0y_s)$ for $s\ge 0$  and extended to $\R$ by symmetry, so $\rho(s,t)=\rho(t-s)$ and similarly for $\rho^{\epsilon}$. We have, for $u>0$ and $H > \f 1 2 $,
\begin{equs}
\rho(u)=\sigma^2 H(2H-1) \int_{-\infty}^{u}\int_{-\infty}^0  e^{-(u-r_1-r_2) }  |r_1-r_2|^{2H-2} dr_1 dr_2.
\end{equs}
We recall the following correlation decay from \cite{Cheridito-Kawaguchi-Maejima},
\begin{lemma}
	\label{correlation-lemma}
	Let  $H\in (0,1) \setminus \{ \f 1 2\}$. Then,
	$\rho(s)= \sigma^2 H(2H-1) s^{2H-2} +O(s^{2H-4})$ as $ s \to \infty$. In particular,  for any $s \geq 0$,
	\begin{equation}\label{cor1}
	|\rho(s)| \lesssim 1\wedge |s|^{2H-2}.
	\end{equation}
\end{lemma}

By Lemma \ref{correlation-lemma},  $  \int_0^\infty \rho^m(s) ds$ is finite if and only if  $ H^*(m)<\f 12$. We are not interested in $H=\f 12$,  as the Ornstein-Uhlenbeck process admits an exponential decay of correlations and $\rho^m$ is integrable for any $m\ge1$. The following estimates explains 
how to choose the appropriate scaling constants, see \cite{Gehringer-Li-fOU} for detail.
\begin{lemma}\label{Integrals}
	Let $H\in (0,1) \setminus \{ \f 1 2\}$ and fix a finite time horizon $T$, then, for $t \in [0,T]$ the following holds \emph{uniformly} for $\epsilon \in (0,\f 1 2]$:
	\begin{equation} \label{correlation-decay-2-1}
	\left( \int_0^{\f t \epsilon} \int_0^{\f t \epsilon}  \vert  \rho(u,r) \vert^m\, dr \,du\right)^{\f 12} \\
	\lesssim 
	\left\{	\begin{array}{lc}
	\sqrt {\f t \epsilon  \int_0^\infty \rho^m(s) ds},  \quad  &\hbox {if} \quad H^*(m)<\f 12,\\
	\sqrt { (\f t \epsilon)  \vert \ln\left(\f 1 \epsilon \right) \vert}, \quad  &\hbox {if} \quad H^*(m)=\f 12,\\ 
	\left(  \f t \epsilon\right) ^{H^*(m)},  \quad &\hbox {if} \quad H^*(m)>\f 12.
	\end{array} \right.
	\end{equation}
	
	\begin{equation} \label{correlation-decay-2-2}
	\left( \int_0^{t} \int_0^{t}  \vert  \rho^{\epsilon}(u,r) \vert^m\, dr \,du\right)^{\f 12} \\
	\lesssim 
	\left\{	\begin{array}{lc}
	\sqrt { t \epsilon  \int_0^\infty \rho^m(s) ds} ,  \quad  &\hbox {if} \quad H^*(m)<\f 12,\\
	\sqrt { t \epsilon  \vert \ln\left(\f 1 \epsilon \right) \vert}, \quad  &\hbox {if} \quad H^*(m)=\f 12,\\
	t\left(  \f t \epsilon\right) ^{H^*(m)-1},  \quad &\hbox {if} \quad H^*(m)>\f 12.
	\end{array} \right.
	\end{equation}	
	In particular,
	\begin{equation}\label{integral-10}
	t \int_0^{t} \vert \rho^{\epsilon}(s)\vert^m ds 
	\lesssim \f { t^{ \left(2H^*(m) \vee 1\right)}}  { \alpha \left( \epsilon, H^*(m)\right)^2}.
	\end{equation}
\end{lemma}
	Note, if $H=\f 12$,  the  bound is always $ \sqrt {\f t \epsilon \int_0^\infty \rho^m(s) ds}$.	

\subsection{Hermite Rank}

We take the Hermite polynomials of degree $m$ to be $H_m(x) = (-1)^m e^{\f {x^2} {2}} \f {d^m} {dx^m} e^{\f {-x^2} {2}}$.  Thus, $H_0(x)=1$,  $H_1(x)=x$. 
The Hermite rank of an $L^2(\mu)$ function with respect to a Gaussian measure
 is the degree of the lowest non-zero Hermite polynomial term in the Hermite polynomial expansion of $G_k$.
\begin{definition} 
Let  $G:\R\to \R$ be an $L^2(\mu)$ function with chaos expansion 
\begin{equation}
G(x)=\sum_{k=m}^\infty c_k H_k(x), \qquad \qquad  c_k=\f 1 {k!}\<G, H_k\>_{L^2(\mu)}.
\end{equation}
\begin{enumerate}
\item The smallest $m$ with $c_m\not =0$ is called the Hermite rank of $G$. 
\item Set $H^*(m_k)=m_k(H-1)+1$.
 If $H^*(m)\le \f 12$ we say $G$ has  {\it high Hermite rank} (relative to $H$),  otherwise it is said to have {\it low Hermite rank}.
\end{enumerate}
\end{definition}

\subsection{Joint functional  CLT / non-CLT}

Functional  limit theorems for Guassian processes  have been extensively studied. 
The theorem we will need is from  \cite{Gehringer-Li-fOU}, it is tailored for proving the lifted functional limit theorem.  
We first introduce the notations.

\begin{convention}
	Let  $y_t^\epsilon=y_{\f t\ \epsilon}$ be the rescaled stationary fractional Ornstein-Uhlenbeck process with standard Gaussian distribution $\mu$ and Hurst parameter $H\in (0,1) \setminus \{ \f 1 2\}$. Each $G_k:\R\to \R$ is a centred function in $L^{2}(\mu)$ with Hermite rank $m_k$. 	Let $ \alpha_k(\epsilon)=\alpha\left(\epsilon,H^*(m_k)\right)$.
	Set 
\begin{equation}\label{x-epsilon}
X^\epsilon:=\left( X^{1,\epsilon}, \dots, X^{N,\epsilon}\right), \qquad \qquad
\hbox{ where } \quad  X_t^{k,\epsilon}= \alpha_k (\epsilon)\int_0^{t} G_k(y^{\epsilon}_s)ds.
\end{equation}
We further define the rough paths $\X^{\epsilon}=( X^\epsilon, \XX^{i,j,\epsilon})$, where
\begin{equation}\label{lifted-x-epsilon}
 \XX^{i,j,\epsilon}_{u,t} := \int_u^t  ( X^{i,\epsilon}_s - X^{i,\epsilon}_u ) dX^{j,\epsilon}_s
= \alpha_i(\epsilon) \alpha_j(\epsilon)\int_u^{t} \int_u^s   G_i(y^{\epsilon}_r) G_j(y^{\epsilon}_s) \,dr ds.
\end{equation}
The process $\X^{\epsilon}=(X^\epsilon, \XX^{i,j,\epsilon})$ is called the canonical lift of   $X^\epsilon$.
	\end{convention}

 Without any further assumptions on $G_k$,  $X^\epsilon$ can be shown to converge jointly in 
 finite dimensional distributions. For the convergence in a H\"older  topology, we assume that $G_k\in L^{p_k}(\mu)$ for  $p_k$  sufficiently large. This means $ H^*(m_k) - \f {1} {p_k} > 0$ if $G_k$ has low Hermite rank and  otherwise $ \f 1 2 - \f {1} {p_k} >0$. 
This condition is summarised in part (3) of  Assumption \ref{assumption-multi-scale}.

\begin{theorem}[Joint Functional CLT/non-CLT]\label{theorem-CLT}
Suppose that $G_k$ are centred and satisfies furthermore  Assumption \ref{assumption-multi-scale} (3). Write  
	$G_k= \sum_{l=m_k}^\infty  c_{k,l} H_l$
and set   $$ X^{W, \epsilon}=\left( X^{1,\epsilon}, \dots, X^{n,\epsilon}\right), \qquad
	X^{Z, \epsilon}=\left( X^{n+1,\epsilon}, \dots, X^{N,\epsilon}\right).$$
	Then, the following holds:
	\begin{enumerate}
		\item 
			There exist stochastic processes $X^W=( X^1, \dots, X^n)$  and $X^Z=(X^{n+1}, \dots, X^N)$ such that on every finite interval $[0,T]$,
			$$(X^{W,\epsilon}, X^{Z,\epsilon}) \longrightarrow  (X^W, X^Z),$$
			weakly  in $\C^{\gamma}([0,T],\R^N)$. We can take $\gamma$ to be any number  smaller than $ \f 1 2 - \f {1} {\min_{k \leq n } p_k}$ if at least one component converges to a Wiener process, otherwise we can take  $\gamma < \min_{k>n} H^*(m_k) - \f 1 {p_k} $.
\item			In particular the following holds, 
			$$\sup_{\epsilon \in (0,\f 1 2)} \left \Vert X^{k,\epsilon}_{s,t} \right\Vert_{p_k} 
			\lesssim  		
			\left\{	   \begin{array}{lc}
			\sqrt {\vert t-s \vert},  \quad  &\hbox {if} \quad H^*(m) \leq \f 12,\\
			\vert t-s\vert^{H^*(m)},  \quad &\hbox {if} \quad H^*(m)>\f 12.
			\end{array} \right.
			$$

			Furthermore, for any $t >0$
			$$\lim_{\epsilon \to 0} \|X^{Z,\epsilon}_t \to X^Z_t\|_{L^2(\Omega)}=0.$$

		\item
		The limit  $X= (X^W,X^Z)$ has the following properties
		\begin{enumerate}
			\item [(1)]   $X^W \in \R^n$  and $X^Z \in \R^{N-n}$ are independent. 
			\item [(2)]  $ X^W = U \hat W_t$ where $ \hat W_t$ is a standard Wiener process and $U$ is a square root of
			the matrix $(2A^{i,j})_{i,j \leq n }$. Let $\rho(r)=\E (y_ry_0)$, then the entries of the matrix are given as follows:
			$$A^{i,j}=\int_0^{\infty} \E\left( G_i(y_s) G_j(y_0) \right) ds = 
			\sum_{q=m_i\vee m_j}^{\infty}   c_{i,q}\; c_{j,q}  \; (k!) \, \int_0^\infty  \rho(r)^q\, dr$$
		 In other words,  $\E\left( X^i_t X^j_s\right)= 2 (t \wedge s) A^{i,j}$ for $i,j\le n$.
			
			\item [(3)] Let $Z_t^{H^*(m_k),m_k}$ be the Hermite processes, represented by (\ref{Hermite}), and 
			\begin{equation}\label{Hermite-2}
			Z_t^{k}= \f{m_k!}{K(H^*(m_k),m_k)} Z_t^{H^*(m_k),m_k}.
			\end{equation}
			Then,
			$$ X^Z=(c_{n+1,m_{n+1}}  Z_t^{n+1} , \dots, c_{N,m_{N}}  Z_t^{N}).$$   	
			We emphasize that the Wiener process defining the Hermite processes is the same for every $k$,
			which is in addition independent of $\hat W_t$. 
		\end{enumerate}
	\end{enumerate}
\end{theorem}

\subsection{Assumptions and Conventions}

 \begin{definition}
	A function $G\in L^2( \mu)$,   $G=\sum_{l=0}^\infty  c_l H_l$,
	is said to satisfy the fast chaos decay condition with parameter $q\in \N$, if 
	$$\sum_{l=0}^\infty  {|c_l|}\; \sqrt{l!} \;(2q-1)^{\f l2}<\infty.$$ 
\end{definition}

For functions $G_1, \dots, G_N$ in $ L^2(\mu)$, we write $m_k$ for their Hermite ranks.

\begin{convention}\label{convention}
	Given a collection of functions $(G_k \in L^2(\mu), k\le N)$, we will label the high rank ones first
	so $H^*(m_k) < \f 12$ for $k=1,\dots, n$, where $n \geq 0 $ and  otherwise  $H^*(m_k) > \f 1 2$. \end{convention}

\begin{assumption}
	[CLT rough, $\FC^\gamma$- assumptions]
	\label{assumption-multi-scale}
	Each $G_k $ belongs to $L^{p_k}(\mu)$ for some  $p_k >2$ and has Hermite rank $m_k\ge 1$. Furthermore,
	\begin{enumerate}
		\item [(1)] Each $G_k$ satisfies the fast chaos decay condition with parameter $q \geq 4 $.
		\item [(2)](Integrability condition) $p_k$ is sufficiently large so the following holds:
		\begin{equation}\label{Hoelder-sum>1}
		\min_{k\leq n} \left( \f 1 2 - \f 1 {p_k} \right) + \min_{n<k\leq N} \left( H^*(m_k) - \f 1  {p_k} \right)>1.
		\end{equation}
		\item[(3)] If $G_k$ has low Hermite rank,  assume $ H^*(m_k) - \f 1 {p_k} > \f 1 2$; otherwise assume $\f 1 2 - \f 1 p_k > \f 1 3$.
		\item[(4)] 
		Either $H^*(m_k) <0$ or $H^*(m_k)> \f 1 2$.
	\end{enumerate}
\end{assumption}

\begin{remark}
	\
	\begin{enumerate}
		
		\item 
		
		If the functions $G_k$ are polynomial functions,  all assumptions stated above  are automatically satisfied, except for (4).
		
		\item 
		The moment assumptions  arise from the necessity to obtain  the convergence, not just in the space of continuous functions but also in a rough path space $\FC^{\gamma}$ for some $\gamma > \f 1 3$, which is naturally established by Kolmogorov type arguments, to be able to use the continuity of the solution maps in the rough path setting.
		
		\item  Let $\eta $ denote the greatest common H\"older continuity exponent for the first $n$ terms in $X^\epsilon$, each of these converge to a Wiener process. Let $\tau$ denote the greatest common H\"older continuity exponent for the rest of the components of  $X^\epsilon$. Then condition (2) is used for making sure 
	$\eta + \tau >1$. With this,  any iterated integral, in which one term converges to a Wiener and the other one to a Hermite process, can be interpreted as a Young integral.

		\item 
		In Condition (4) we have to assume   $H^*(m_k)<0$, leaving a gap  $[0, \f 12]$.
	This restriction is due to Proposition \ref{integrable-lemma}, where we  only
		obtain the required integrability estimates for $H^*(m_k)<0$.  	
		\end{enumerate}
\end{remark}

\section{Lifted joint functional  limit theorem}
\label{lifted-clt}
If $X^{(n)}$ and $Y^{(n)}$ are two sequences of stochastic processes with $X^{(n)}\to X$ and $Y^{(n)}\to Y$ 
(even if the convergence is almost surely everywhere and even if $X$ and $Y$ are differentiable curves), we may fail to conclude that 
$\int_0^t X^{(n)}_s dY^{(n)}_s ds \to \int _0^t X_s dY_s$. Take for example $X_t^{(n)}=\f 1{\sqrt n} \cos(nt)$ and 
$Y_t^{(n)}=\f 1 {\sqrt n}\sin(nt)$. 
If a sequence of vector valued stochastic processes  $(X_{1}^{(n)},X_2^{(n)})$ together with its canonical lift converge in the rough path topology, the limit of the iterated integrals may not  be the same as the iterated integrals of the limit.
We give an example for this by modifying the earlier example by pumping randomness into the $\cos$ and $\sin$ sequences using random variables $\lambda(1), \lambda(2) $ taking values in $\{1, -1\}$. Define a  sequence of stochastic processes $\{X^{(n)}_{1}\}$ as follows:
$$X_{1}^{(n)}(t)=\left\{\begin{array}{ll}
\f 1{\sqrt n} \cos(nt),   \qquad \lambda(1)=1, \\
\f 1{\sqrt n} \sin(nt),   \qquad \lambda(1)=-1,
\end{array}  \right. $$
and similarly $X_2^{(n)}$.
Then, $X_{1}^{(n)} (s)\to 0$ in $\C^\alpha$ for $\alpha<\f 12$ and the same holds true for $X_2^{(n)}$, however,
$$\int_0^t X_{1}^{(n)} (s) dX_{2}^{(n)}(s) =
\left\{\begin{array}{cl}
\f t2,   \qquad  &\lambda(1)=1,  \lambda(2)=-1, \\
0,   \qquad &\lambda(1) = \lambda(2), \\
-\f t2,   \qquad &\lambda(1)=-1, \lambda(2)=1. 
\end{array}  \right. $$
In this example,  $(X_{1}^{(n)},X_2^{(n)})$ together with its canonical lift converge in the rough path topology.
The limit of the iterated integrals  depend on $\lambda$. If we set $\lambda$ so that
$(\lambda(1), \lambda(2))$ is uniformly distributed, the marginals are always the same, but the joint distributions depends on the further correlation relations of the random variables $\lambda(1)$ and $\lambda(2)$.

In this section, we show that $\X^{\epsilon}=(X^\epsilon, \XX^{i,j,\epsilon})$, the canonical lift of $X^\epsilon$,   converges in the rough path topology. 
Specifically, we will show in \S\ref{sub-section:Ito} that the secondary processes $ \XX^{i,j,\epsilon}$, involving only $i,j\le n$, converge jointly in finite dimensional distributions (which is more involved due to the lack of the strong mixing property). In \S\ref{tightness},
we prove that  $\{ (X^\epsilon, \XX^{i,j,\epsilon}), \epsilon\in (0, \f 12]\} $ is tight in the rough path topology.
The tightness plus the fact that we can identify the limiting joint probability distributions with stochastic integrals $\int_0^t X_i d X_j$  shows that $( X^{W, \epsilon},  \XX^{i,j,\epsilon}, i,j\le n)$ converges in the rough path to $X^W$ and its lift. Furthermore we identify its remaining canonical lift parts of  $(X^Z, X^W)$ as a measurable functions of $(X^W,X^Z)$.  The rest follows from Theorem \ref{theorem-CLT}.

\subsection{Relative compactness of iterated integrals}
\label{tightness}
In this section, we establish moment bounds on the iterated integrals and prove that $\X^\epsilon$ is tight in the rough path topology. Let $G_i$ and $ G_j$ be two functions in $L^2(\mu)$ with Hermite ranks $m_{G_i}$ and $m_{G_j}$ respectively. Set $\alpha_i=\alpha(\epsilon,H^*(m_{G_i}))$ and $\alpha_j(\epsilon)=\alpha(\epsilon,H^*(m_{G_j}))$.
Recall that $$
\XX^{i,j,\epsilon}_{u,t}
=\alpha_i(\epsilon) \alpha_j(\epsilon)\int_u^{t} \int_u^s G_i(y^{\epsilon}_r) G_j(y^{\epsilon}_s) dr ds,
$$
To obtain tightness, we assume that the coefficients $c_{n,i}$  in the Hermite expansion of  $G_i$ satisfy the decay condition specified in Assumption \ref{assumption-multi-scale} (1). 
 We want to argue by Theorem 3.1 in \cite{Friz-Hairer}, the rough path analogue to Kolmogorov's theorem. Thus, we need to  estimate $ \Vert \XX^{i,j,\epsilon}_{u,t} \Vert_{L^p(\Omega)}$, where by stationarity we may from now on assume $u=0$. 

If $G_i$ and $G_j$ are in a  finite chaos of order $Q$,  then
\begin{align}\label{equation-expansion-iterated-integrals-finite-chaos}
&\E \left( \XX^{i,j,\epsilon}_{0,t} \right)^p= \E \left( \alpha_i(\epsilon) \alpha_j(\epsilon) \int_0^t\! \int_0^s G_i(y^{\epsilon}_r) G_j(y^{\epsilon}_s) dr ds \right)^p \\
&= \alpha_i(\epsilon)^p \alpha_j(\epsilon)^p \E \left(  \int_0^t \int_0^s \sum_{k,k'=1}^Q   c_{i,k} c_{j,k'} H_k(y^{\epsilon}_r) H_{k'}(y^{\epsilon}_s)  dr ds  \right)^p\\
&\leq \alpha_i(\epsilon)^p \alpha_j(\epsilon)^p \sum^{Q}_{k_1, \dots, k_p = m_{G_i}  } \sum_{ k'_1, \dots, k'_p = m_{G_j}}^{Q} \prod_{l=1}^{p} \vert c_{i,k_l} c_{j,k'_l} \vert \left \vert  \overbrace{\int_0^t\! \int_0^{s_1} \dots \int_0^t\! \int_0^{s_p} }^{p}   \E \left( \prod_{l=1}^{p} H_{k_l}(y^{\epsilon}_{r_l}) H_{k'_l}(y^{\epsilon}_{s_l}) \right) dr_l ds_l \right \vert.
\end{align}

This means we need to estimate the terms $\E \left( \prod_{l=1}^{p} H_{k_l}(y^{\epsilon}_{r_l}) H_{k'_l}(y^{\epsilon}_{s_l})\right)$. For convenience, we will re-label the indices so to write the product in the form  $\E \left( \prod_{l=1}^{2p} H_{k_l}(y^{\epsilon}_{s_l}) \right)$. For $p=2$, we have the identity $\E(H_m(y^{\epsilon}_s) H_n(y^{\epsilon}_r)) = \delta_{n,m} \left( \E(y^{\epsilon}_s y^{\epsilon}_r) \right)^m$. For the multiple product, 
 we use the so called diagram-formulae, see e.g.\cite{BenHariz} and references therein. The diagram-formulae formula states that the expectation we are concerned with can be calculated by summing over products of covariances, similar to Isserli's/Wick's theorem.
This can be linked to graphs. Nodes of these graphs correspond to the $y^{\epsilon}_{s_l}$'s and each such node has exactly $k_l$ edges, where no edge may connect a node to itself. Each edge between $y^{\epsilon}_{s_l}$ and $y^{\epsilon}_{s_q}$ corresponds to a  factor $ \E ( y_{s_q} ^\epsilon y_{s_l} ^\epsilon )$.
The expectation we are concerned with is then given by summing over all possible graphs of such complete parings.

For a particular graph $\Gamma$, we denote by $n(l,q)$ the number of edges connecting $l$ to $q$, so it takes values in $\{0,1,\dots ,\min (k_l,k_q)\}$,
and consider the pairings in an ordered way so that each pairing is counted only once.  We thus have $\sum_{q=1}^{2p} n(l,q)= {k_l}$ and, since edges are only allowed to connect with different nodes $n(q,q)=0$ for every $q$. For any given graph this is
$$\prod_{q=1}^{2p}  \prod_{l=q+1}^{2p} \left( \E(y^{\epsilon}_{s_q}y^{\epsilon}_{s_{l}}) \right)^{n(l,q)} =\prod_{q=1}^{2p} \prod _{\{l: l>q, \, l\in \Gamma_q\}}  \rho^{\epsilon}(s_l-s_q)^{n(l, q)},$$
where $\Gamma_q$ denotes the subgraph of nodes connected to $q$. Thus,
$$ \E \left( \prod_{l=1}^{2p} H_{k_l}(y^{\epsilon}_{s_l}) \right) = \sum_{\Gamma}  \prod_{q=1}^{2p} \prod _{\{l: l>q, \, l\in \Gamma_q\}}  \rho^{\epsilon}(s_l-s_q)^{n(l, q)},
$$
where the sum ranges over all suitable graphs $\Gamma$ given $(k_1, \dots ,k_{2p})$.

\begin{lemma}\label{basic-graph}
	\
	\begin{enumerate}
		\item 	Let $\Gamma$ denote a complete pairing of $2p$ nodes with a suitable amount of edges $(k_1, \dots , k_{2p})$.  
		Define:
		\begin{align*}
		I(\epsilon,2p,\Gamma)&:= \overbrace{ \int_0^{t} \dots  \int_0^{t} }^{2p}
		\prod_{\{  (s_q, s_l) \} \in \Gamma} ( \E(y^{\epsilon}_{s_q} y^{\epsilon}_{s_l}))^{n(q,l)} ds_1\dots ds_{2p}.
		\end{align*}
		Then,
		\begin{equation}
		I(\epsilon,2p,\Gamma) 	
		\lesssim \; \prod_{l=1}^{2p}  
		\sqrt{ t \int_{-t} ^{t}  | \rho^{\epsilon}(s)| ^{k_l}\; ds }
		\lesssim  \prod_{l=1}^{2p} \f{ t^{ H^*(k_l) \vee \f 1 2} } {\alpha \left(\epsilon, H^*(k_l) \right)}.
		\end{equation}
		\item If $G_i, G_j: \R\to \R$ are functions in  finite chaos with Hermite ranks $m_{G_i}$ and $m_{G_j}$ respectively.
		Then,
		\begin{align*}
		\Vert \XX^{i,j,\epsilon}_{0,t} \Vert_{L^p(\Omega)}&=  \alpha_i ( \epsilon)
		\alpha_j ( \epsilon)\,\left\| \int_0^{t} \int_0^s G_i(y^{\epsilon}_r) G_j(y^{\epsilon}_s) dr ds\right\|_{L^p(\Omega)}\\
		&\lesssim t^{H^*(m_{G_i}) \vee \f 1 2 + H^*(m_{G_j}) \vee \f 1 2}.
		\end{align*} 
		
	\end{enumerate}
\end{lemma} 
\begin{proof}
	For a general graph, let us start dealing with the first variable $s_1$. We first count \emph{forward} and observe
	$$ \prod_{\{  (s_q, s_l) \} \in \Gamma}  (\E(y^{\epsilon}_{s_q} y^{\epsilon}_{s_l}))^{n(q,l)}=\prod_{q=1}^{2p}  \prod_{l=q+1}^{2p} (\E(y^{\epsilon}_{s_q}y^{\epsilon}_{s_{l}}))^{n(q,l)} 
	=\prod_{q=1}^{2p} \prod _{\{l: l>q, \, l\in \Gamma_q\}} (\rho^{\epsilon}(s_l-s_q))^{n(l, q)},$$
	where $\Gamma_q$ denotes the subgraph of nodes  connected to $q$.
	Using H\"older's inequality we obtain
	\begin{align*} 
	\int_0^{t}  \prod _{\{q:q >1, \, q \in \Gamma_1\}} | \rho^{\epsilon}(s_1-s_q)|^{n(1, q)}ds_1 
	&\le   \prod _{\{q: q >1, \, q \in \Gamma_1\}}\left(  \int_0^{t}  |\rho^{\epsilon}(s_1-s_q)|^{k_1}\, d s_1  \right)^{ \f {n(1, q)} {k_1} } \\
	&\le  \int_{-t} ^{t}  | \rho^{\epsilon}(s_1)| ^{k_1} ds_1.
	\end{align*}	
	We have used $\sum_{\{ q>1:  q \in \Gamma_1\}} n(1,q)=k_1$, the number of edges at node $1$. 
	We then peel off the integrals layer by layer, and proceed with the same
	technique to the next integration variable.
	For example suppose the remaining integrator containing $s_{2}$ has the combined exponent $\tau_2=\sum_{q=2}^{2p} n(2,q)$, ($\tau_1=k_1$). By the same procedure as for $s_1$ we score a factor
	$$ \int_{-t} ^{t}  | \rho^{\epsilon}(s_2)| ^{\tau_{2}} ds_2. 
	$$
	By induction and putting the estimates for each integral together,
	$$ \overbrace{ \int_0^{t} \dots  \int_0^{t} }^{2p}
	\prod_{q=1}^{2p} \prod _{\{l: l>q, \, l\in \Gamma_q\}} (\rho^{\epsilon}(s_l-s_q))^{n(l, q)}\;ds_1\dots ds_{2p} \lesssim  \prod_{q=1}^{2p}  \int_{-t} ^{t}  | \rho^{\epsilon}(s)| ^{\tau_q} ds.$$
	Following \cite{BenHariz},  we reverse the procedure in the estimation for the integral kernel. Let $\xi_q$ denote the number of edges connected to the
	node $q$ in the backward direction, so $\xi_{q}=\sum_{l=1}^{q} n(l,q)$, and the same reasoning leads to the following estimate:
	$$  \overbrace{ \int_0^{t} \dots  \int_0^{t} }^{2p}
	\prod_{q=1}^{2p} \prod _{\{l: l< q, \, l\in \Gamma_q\}} (\rho^{\epsilon}(s_l-s_q))^{n(l, q)}\;ds_1\dots ds_{2p}
	\lesssim  \prod_{q=1}^{2p}  \int_{-t} ^{t}  | \rho^{\epsilon}(s)| ^{\xi_q } ds.$$
	Since $\tau_q+\xi_q=k_q$  by H\"olders inequality,	$$  \int_{-t} ^{t}  | \rho^{\epsilon}(s)| ^{\tau_q } ds \int_{-t} ^{t}  | \rho^{\epsilon}(s)| ^{\xi_q } ds  
	\le  2 t  \int_{-t} ^{t}  | \rho^{\epsilon}(s)| ^{k_q } ds.$$
	Therefore,
	\begin{align*}\label{single-graph}
	\left( \overbrace{ \int_0^{t} \dots  \int_0^{t} }^{2p}
	\prod_{q=1}^{2p} \prod _{\{l: l>q, \, l\in \Gamma_q\}} (\rho^{\epsilon}(s_l-s_q))^{n(l, q)}\;ds_1\dots ds_{2p}\right)^2
	&   \lesssim  
	\; \prod_{q=1}^{2p} \left( t  \int_{-t} ^{t}  | \rho^{\epsilon}(s)| ^{k_q } ds  \right).
	\end{align*}
	By Lemma \ref{Integrals} we obtain, for each $q \in \{1, \dots, N \}$,
	$$ \alpha \left(  \epsilon, H^*(k_q)\right)^2   t
	\int_{-t}^{t}  | \rho^{\epsilon}(s)| ^{k_q} ds  \lesssim  t^{2 H^*(k_q) \vee  1} ,$$
	hence, the first part of the lemma follows.	

For $G_i$ and $G_j$ we obtain as in Equation (\ref{equation-expansion-iterated-integrals-finite-chaos}), using the fact that $\rho^{\epsilon} >0$ and thus we may enlarge our integration area, 

	\begin{align*}
	&\Vert \XX^{i,j,\epsilon}_{0,t} \Vert^p_{L^p(\Omega)} \\
	&\leq \alpha_i(\epsilon)^p \alpha_j(\epsilon)^p \sum^{Q}_{k_1, \dots, k_p = m_{G_i}  } \sum_{ k'_1, \dots, k'_p = m_{G_j}}^{Q} \prod_{l=1}^{p} \vert c_{i,k_l} c_{j,k'_l} \vert \left \vert  \overbrace{\int_0^t \int_0^{s_1} \dots \int_0^t \int_0^{s_p} }^{p}   \E \left( \prod_{l=1}^{p} H_{k_l}(y^{\epsilon}_{r_l}) H_{k'_l}(y^{\epsilon}_{s_l}) \right) dr_l ds_l \right \vert\\
	&\lesssim  \alpha_i(\epsilon)^p \alpha_j(\epsilon)^p \sum^{Q}_{k_1, \dots, k_p = m_{G_i}  } \sum_{ k'_1, \dots, k'_p = m_{G_j}}^{Q} \prod_{l=1}^{p} \vert c_{i,k_l} c_{j,k'_l} \vert \int_{[0,t]^{2p}}  \E \left( \prod_{l=1}^{p} H_{k_l}(y^{\epsilon}_{r_l}) H_{k'_l}(y^{\epsilon}_{s_l}) \right) dr_l ds_l.\\
	&= \alpha_i(\epsilon)^p \alpha_j(\epsilon)^p \sum^{Q}_{k_1, \dots, k_p = m_{G_i}  } \sum_{ k'_1, \dots, k'_p = m_{G_j}}^{Q} \prod_{l=1}^{p} \vert c_{i,k_l} c_{j,k'_l} \vert \sum_{\Gamma} I(\epsilon,2p,\Gamma)\\
	&\lesssim  \prod_{l=1}^{p}  t^{ H^*(k_l) \vee \f 1 2}   t^{ H^*(k'_l) \vee \f 1 2}.
	\end{align*}
	By monotonicity of $H^*$ and the fact that $k_l \geq m_{G_i}$ and  $k'_l \geq m_{G_j}$,
	$$ \left ( \prod_{l=1}^{p}  t^{ H^*(k_l) \vee \f 1 2}   t^{ H^*(k'_l) \vee \f 1 2} \right)^{\f 1  p } \leq  t^{H^*(m_{G_i}) \vee \f 1 2 + H^*(m_{G_j}) \vee \f 1 2},$$
	concluding the proof.
\end{proof}

For functions not belonging to a finite chaos we must count the number of graphs in the computation  and need some assumptions. 
Let $M(k_1, \dots, k_{2p})$ denote the cardinality of  admissible graphs with $2p$ nodes with respectively $(k_1,\dots , k_{2p})$ edges.
In \cite{Graphsnumber} it was shown that 
$$ M\left(k_1,k_2, \dots, k_{2p} \right) \leq \prod_{l=1}^{2p} (2p-1)^{\f {{k_l}} {2}} \sqrt{k_l}.$$
This leads to Assumption \ref{assumption-multi-scale} (1), which restricts the $G_i$'s to the class of functions whose coefficients in the Hermite expansion decay sufficiently fast.

\begin{proposition}\label{tightness-lemma}
	Suppose that each $G_k$ satisfies Assumption \ref{assumption-multi-scale}. 
	Then,  one has for $i,j \in \{ 1, \dots ,N \}$,
	$$   \left\Vert \alpha_i(\epsilon) \alpha_j(\epsilon) \int_0^{t}\!\!\! \int_0^s G_i(y^{\epsilon}_r) G_j(y^{\epsilon}_s) dr ds \right\Vert_{L^p(\Omega)} \lesssim t^{H^*(m_{G_i}) \vee \f 1 2 + H^*(m_{G_j}) \vee \f 1 2}.$$
	Consequently, $\X^\epsilon$ is tight in $\FC^{\gamma}$ for $\gamma \in (\f 1 3, \f 1 2 - \f {1} {\min_{k \leq n } p_k})$.
\end{proposition}
\begin{proof} 	As above using $\rho^{\epsilon}>0$ and $  \prod_{l=1}^{p}  t^{ H^*(k_l) \vee \f 1 2}   t^{ H^*(k'_l) \vee \f 1 2} \leq t^{p \left( H^*(m_{G_i}) \vee \f 1 2 + H^*(m_{G_j}) \vee \f 1 2 \right)} $,
		\begin{align*}
	&\Vert \XX^{i,j,\epsilon}_{0,t} \Vert^p_{L^p(\Omega} \\
	&\lesssim  \alpha_i(\epsilon)^p \alpha_j(\epsilon)^p \sum^{\infty}_{k_1, \dots, k_p = m_{G_i}} \sum_{ k'_1, \dots, k'_p = m_{G_j}}^{\infty} \prod_{l=1}^{p} \vert c_{i,k_l} c_{j,k'_l} \vert \int_{[0,t]^{2p}}  \E \left( \prod_{l=1}^{p} H_{k_l}(y^{\epsilon}_{r_l}) H_{k'_l}(y^{\epsilon}_{s_l}) \right) dr_l ds_l.\\
	&= \alpha_i(\epsilon)^p \alpha_j(\epsilon)^p \sum^{\infty}_{k_1, \dots, k_p = m_{G_i}} \sum_{ k'_1, \dots, k'_p = m_{G_j}}^{\infty} \prod_{l=1}^{p} \vert c_{i,k_l} c_{j,k'_l} \vert \sum_{\Gamma} I(\epsilon,2p,\Gamma)\\
		&\lesssim t^{p \left( H^*(m_{G_i}) \vee \f 1 2 + H^*(m_{G_j}) \vee \f 1 2 \right)} \sum^{\infty}_{k_1, \dots, k_p = m_{G_i}  } \sum_{ k'_1, \dots, k'_p = m_{G_j}}^{\infty} \prod_{l=1}^{p} \vert c_{i,k_l} c_{j,k'_l} \vert  M(k_1, \dots ,k_p , k'_1 , \dots , k'_p) \\
	&\lesssim t^{p \left( H^*(m_{G_i}) \vee \f 1 2 + H^*(m_{G_j}) \vee \f 1 2 \right)} \sum^{\infty}_{k_1, \dots, k_p = m_{G_i}  } \sum_{ k'_1, \dots, k'_p = m_{G_j}}^{\infty} \prod_{l=1}^{p} \vert c_{i,k_l} c_{j,k'_l} \vert \sqrt{k_l! k'_l!} (2p-1)^{\f {k_l + k'_l} {2}}.
	\end{align*}
	By the chaos decay assumption, these  sums  are finite and this completes the proof for the required moment bounds. 
Finally, using Theorem \ref{theorem-CLT}, we can conclude the tightness of $\X^\epsilon$ in $\FC^{\gamma}$, where $\gamma \in (\f 1 3 , \f 1 2 - \f {1} {\min_{k\leq n} p_k})$,  by an application of Lemma \ref{tightness-second-order}.
	\end{proof}

\subsection{Young integral case  (functional non-CLT in rough topology)}
\label{sub-section:Young}

\begin{lemma}\label{young-lift}
	Assume Assumption \ref{assumption-multi-scale}.Then,
	\begin{equation}\label{lift-Hermite}
	(X^{\epsilon}, \; \XX^{i,j,\epsilon})_{ \{ i,j \in  \{1, \dots ,N  \}: i \vee j >n  \} },
	\end{equation}
	converges in  finite dimensional distributions to $(X,\XX^{i,j})$, where $\XX^{i,j}=\int_0^t X_s^i dX_s^j$  and these integrals are well defined as Young integrals. 

\end{lemma}
\begin{proof}
	By Assumption \ref{assumption-multi-scale} and Theorem \ref{theorem-CLT}, each component of $X^{\epsilon}$ converges in a H\"older space. Furthermore, by Assumption \ref{assumption-multi-scale} (2) there exist numbers $\eta$ and $\tau$, with $\eta + \tau >1$, such that the H\"older regularity of the  limits corresponding to a Wiener processes, are bounded below  by $\eta$,
	and the ones corresponding to a Hermite process  bounded from below by $\tau$. Therefore,   taking the integrals
	$$ 
	\alpha_i(\epsilon) \alpha_j(\epsilon)\int_0^{t} \int_0^s G_j(y^{\epsilon}_s) G_i(y^{\epsilon}_r) dr ds = \int_0^t X^{i,\epsilon}_s dX^{j,\epsilon}_s
	$$ is a continuous and well-defined operation from $\C^\eta\times \C^\tau\to \C^\tau$ or $\C^\tau \times \C^\eta \to \C^\eta$ , thus weak convergence in $\C^{\eta}$  follows. 
Let $F$ denote the continuous map such that, for $i,j$ with $i\vee j > n$,  $\XX^{i,j,\epsilon} = F(X^{\epsilon})^{i,j}$. Now, set 
\begin{align*}
\mathfrak{F}&= \id \times F\\
\mathfrak{F}(X^{\epsilon}) &=(X^{\epsilon},F(X^{\epsilon}))=(X^{\epsilon},\XX^{i,j})_{ \{ i,j \in  \{1, \dots ,N  \}: i \vee j >n  \} },
\end{align*} 
which by the above is a continuous function. Thus, by an application of the continuous mapping theorem we can conclude the lemma.
\end{proof}

\begin{remark}
Note that by the moment bounds obtained in Theorem \ref{theorem-CLT} and Proposition \ref{tightness-lemma} the joint convergence takes place in better H\"older spaces.
\end{remark}

Now it is left to deal with the parts of the natural rough path lift involving  two Wiener scaling terms, this is carried out in the next section.

\subsection{It\^o integral case (functional CLT in rough topology)}
\label{sub-section:Ito}
We  proceed to establish the convergence of  the iterated integrals where both components belong to the high Hermit rank case.
 \begin{remark}
We further assume $H^*(m_k) < 0 $ for each $k$ which gives rise to a Wiener scaling. Thus, we do not obtain Logarithmic terms and therefore work with the $\f {1} {\sqrt \epsilon}$ scaling from here on. Furthermore, in this case $\alpha(\epsilon) \int_0^t G(y^{\epsilon}_s) ds$ equals $\sqrt{\epsilon} \int_0^{\f t \epsilon} G(y_s) ds$ in law and for simplicity we will work with the latter in this chapter.
\end{remark}
From here onwards in this section, we take $k,i,j \leq n$. Thus, both  $G_i$ and $G_j$ give rise to Wiener processes. Recall that,
 $$X^{k,\epsilon}_t= \sqrt{\epsilon} \int_0^{\f t \epsilon} G_k(y_s) ds.$$
By Theorem \ref{theorem-CLT},   $( X^{i,\epsilon} ,  X^{j,\epsilon} ) \to (W^i,W^j)$, where $W^i$ and $W^j$ denote Wiener processes with covariances as specified in Theorem \ref{theorem-CLT}, weakly.
We now want to show that the convergence of the following integral
\begin{align*}
\int_0^{\f t \epsilon}   X^{i,\epsilon}_s dX^{j,\epsilon}_s &= \epsilon \int_{0}^{\f t \epsilon} \int_0^s  G_i(y_r) G_j(y_s) dr ds\\
&= I_1(\epsilon) + I_2(\epsilon).
\end{align*} 
We will show that $I_1(\epsilon) \to \int_0^t W^i_s dW^j_s $ weakly, where the integral is understood in the It\^{o}-sense, and $I_2(\epsilon) \to t A^{i,j}$ in probability for some constants $A^{i,j}$.
For this we aim to use \cite{Kurtz-Protter} Theorem 2.2 , hence, we need to approximate  $X^{k,\epsilon}$  by a suitable martingale, see also \cite{roughflows}. For any $L^2(\mu)$ function $U$, in particular for the $G_k$'s, one would have liked to work with the stationary process,
$$\begin{aligned}\Phi_U(t)&= \int_t^\infty U(y_r) dr
\end{aligned}$$
and use it to define $L^2(\Omega)$-martingale differences, see \cite{Kipnis-Varadhan}. 
This unfortunately does not posses good enough integrability properties, thus, as in \cite{roughflows}, we instead define
\begin{equation} \begin{aligned}\hat U(k):= \int_{k-1}^\infty  \E( U(y_r)\,|\, \F_k) \, dr. \\
\end{aligned}
\end{equation}
Since  $y$ is  stationary, we do have $(\hat U \circ \tau)(k) =\hat U({k+1})$, where $\tau$ is the shifting operator on sequences. 
To show that $\hat{U}$ posses the desired integrability properties is a bit more involved.
We will show that that there exists a local independent decomposition of the fractional Ornstein-Uhlenbeck process as follows: for every $t$ there exists a decomposition,
$y_t = \overline{y}^k_t + \tilde{y}^k_t,$
such that the first term $\overline{y}^k_t$ is $\mathcal{F}_k$ measurable, $\tilde{y}^k_t$ is independent of  $\mathcal{F}_k$, where $\F_k$ is  the filtration generated by the driving fractional Brownian motion up to time $k$.   Both terms are Gaussian processes.
This is given in section \ref{prove-lemma-int}. To proceed further we also need a couple of lemmas.
 
 \begin{lemma}
For $x,y,a,b \in \R$ such that $a^2+b^2=1$,
\begin{equation}
H_m(ax+by) = \sum_{j=0}^m  \binom{m}{j} a^{j} b^{m-j} H_{   j}(x) H_{m-   j}(y).
\end{equation}
\end{lemma}

\begin{lemma}
Let $H \in (0,1)\setminus \{ \f 1 2 \}$. Set $a_t= \Vert \overline{y}^k_t \Vert_{L^2(\Omega)} $. Then,
$$\E[ H_m(y_t)|\mathcal{F}_k] = (a_t)^m  H_m\left( \f  {\overline{y}^k_t}  {a_t} \right).$$
\end{lemma}
\begin{proof}
Let  $y_t = \overline{y}^k_t + \tilde{y}^k_t$ denote the local independent decomposition
of the fOU from \ref{prove-lemma-int} and set  $b_t=\Vert  \tilde{y}^k_t \Vert_{L^2(\Omega)} $.
By the independence of $\overline{y}^k_t $ and $ \tilde{y}^k_t$ we obtain 
\begin{align*}
1 &=\Vert y_k \Vert_{L^2(\Omega)}^2 = \Vert \overline{y}^k_t \Vert_{L^2(\Omega)}^2 + \Vert  \tilde{y}^k_t \Vert_{L^2(\Omega)}^2
= (a_t)^2 + (b_t)^2.
\end{align*}
 Now we decompose $H_m(y_t)$ using the above identity and obtain,
\begin{align*}
H_m(y_t)&=H_m\left( \overline{y}^k_t + \tilde{y}^k_t \right)
=H_m\left( a_t \left(  \f  {\overline{y}^k_t}  {a_t} \right) + b_t \left( \f {\tilde{y}^k_t}  {b_t} \right) \right)\\
&= \sum_{j=0}^m  \binom{m}{j} a_t^j  b_t^{m-j}  H_j\left( \f  {\overline{y}^k_t}  {a_t} \right)    H_{m-j}\left( \f {\tilde{y}^k_t}  {b_t} \right).
\end{align*}
By construction $\f  {\overline{y}^k_t}  {a_t}$ and $ \f {\tilde{y}^k_t}  {b_t}$ are standard Gaussian random variables, together with the fact that  $ \bar y^k_t$ is  measurable with respect to $\F_k$ this leads to, 
\begin{align*}
\E[ H_m(y_t)|\mathcal{F}_k] &= \sum_{j=0}^m \binom{m}{j} (a_t)^j  (b_t)^{m-j} 
H_j\left( \f  {\overline{y}^k_t}  {a_t} \right)  \E\left [H_{m-j} \left( \f {\tilde{y}^k_t}  {b_t} \right)| \mathcal{F}_k\right]\\
&= (a_t)^m  H_m\left( \f  {\overline{y}^k_t}  {a_t} \right),
\end{align*}
where we  used  the fact that 
$\E \left(H_j\left( \f {\tilde{y}^k_t}  {b_t} \right)\right)$ vanishes for any $j\ge 1$,   $ \tilde y^k_t$ is  independent of $\F_k$, and $H_0=1$.
\end{proof}

\begin{proposition}\label{prop-integrability}
If $U\in L^2(\mu)$ has Hermite rank $m$, then 
\begin{equation}\label{integrable-1}
\Vert \hat U(k)\Vert_{L^2(\mu} \leq \|U\|_{L^2(\mu)}\int_{k-1}^{\infty} \int_{k-1}^{\infty}  \left( \E \left(    {\overline{y}^k_s} {\overline{y}^k_r}  \right) \right)^{ m }  dr \, ds.\end{equation}
In particular  $\{\hat U(k)\}_{k\ge 1}$ is bounded in  $L^2(\Omega)$ if  $H \in (0,1)\setminus \{ \f 1 2 \}$ and  $U$ has Hermite rank  $m$ such that $H^*(m)<0$.
\end{proposition}

\begin{proof}
The `in particular' part of the assertion follows from the statement that
if  $ H \in (0,1)\setminus \{ \f 1 2 \}$ and  $U$ has Hermite rank  $m$ such that $H^*(m)<0$, then, $\int_{k-1}^{\infty} \int_{k-1}^{\infty}  \, \left( \E \left(    {\overline{y}^k_s} {\overline{y}^k_r}  \right) \right)^q  dr \, ds < \infty$, see Proposition \ref{integrable-lemma}.
Due to the lack of the strong mixing property, the proof for this is lengthy and independent of the error estimates here and therefore postponed to section \ref{prove-lemma-int}. 

We go ahead proving the identity.  Starting with the definition of $\hat U$ and the Hermite expansion  $U=\sum_{q=m}^{\infty} c_q H_q$, we compute the $L^2(\Omega)$ norm as follows: 
\begin{align*}  
\Vert \hat U(k)\Vert_{L^2(\Omega)}
&= \int_{k-1}^{\infty} \int_{k-1}^{\infty} \sum_{q=m}^{\infty} \sum_{j=m}^{\infty} c_q c_j \E\Big( \E[ H_q(y_s)| \mathcal{F}_k] \,\E[H_j(y_r)| \mathcal{F}_k]  \Big) dr \, ds\\
&= \int_{k-1}^{\infty} \int_{k-1}^{\infty} \sum_{q=m}^{\infty} (c_q)^2   \E \left( (a_s)^q (a_r)^q   H_q\left( \f  {\overline{y}^k_s}  {a_s} \right) H_q\left( \f  {\overline{y}^k_r}  {a_r} \right) \right)  dr \,ds\\
&= \int_{k-1}^{\infty} \int_{k-1}^{\infty} \sum_{q=m}^{\infty} (c_q)^2  \,q! \,(a_s)^q  (a_r)^q \left( \E \left(  \f  {\overline{y}^k_s}  {a_s}  \f  {\overline{y}^k_r}  {a_r}  \right)\right)^q   dr\, ds\\
&=\int_{k-1}^{\infty} \int_{k-1}^{\infty} \sum_{q=m}^{\infty} (c_q)^2 \, q! \,\left(  \E \left(    {\overline{y}^k_s} {\overline{y}^k_r}  \right) \right)^q  dr \, ds \leq \|U\|_{L^2(\mu)}
\int_{k-1}^{\infty} \int_{k-1}^{\infty}  \, \left( \E \left(    {\overline{y}^k_s} {\overline{y}^k_r}  \right) \right)^{ m }   dr \, ds .
\end{align*}
The desired conclusion follows from the summability of  $ \sum_{q=m}^{\infty} (c_q)^2 \, q!$, which is $\|U\|_{L^2(\mu)}^2$.
\end{proof}
With this we may define two families of $L^2(\Omega)$ martingales.
\begin{corollary}\label{cor-integrability}  
Given $U, V \in L^2(\mu)$ such that there Hermite ranks $m_U$ and $m_V$ satisfy $H^*(m_U)<0$ and $H^*(m_V)<0$, then, the process  $(M_k, k\ge 1)$, where
	$$M_k :=\sum_{j=1}^{k}\left( \hat U(j)   -\E\left  (\hat U(j) |\F_{j-1} \right) \right),$$
	is an $\F_k$-adapted  $L^2(\Omega)$  martingale with shift covariant martingale difference. The same holds for 
	$$ N_k := \sum_{j=1}^k \left( \hat V(j) - \E ( \hat V(j) | \mathcal{F}_{j-1}) \right).$$
\end{corollary}

We can now formulate the main result of this sub-section.
\begin{proposition}\label{prop-area}
Let $U,V,M$ and $N$  be as in Corollary \ref{cor-integrability}, then 
there exists a function $\err(\epsilon)$ converging to zero in probability as $\epsilon\to 0$,  such that
\begin{equation}\label{area1-2} \begin{aligned}
&\epsilon \int_{0}^{\f t \epsilon} \int_0^s  U(y_s) V(y_r) dr ds 
&=\epsilon \sum _{k=1} ^{[\f t \epsilon]} (M_{k+1}-M_k)N_k +  t \gamma+\err(\epsilon),
\end{aligned}
\end{equation}
where $$\gamma
= \int_0^\infty \E (U(y_s) V(y_0) )ds.$$
\end{proposition}
The proof for this is given in the rest of the section.  
Afterwards we show that $\epsilon \sum _{k=1} ^{[\f t \epsilon]} (M_{k+1}-M_k)N_k$ converges to  the relevant It\^o integrals of the limits of
$\sqrt \epsilon \int_0^{[\f t \epsilon]} U(y_r) dr$ and $\sqrt \epsilon \int_0^{[\f t \epsilon]} V(y_r) dr$.

\begin{lemma}
The stationary Ornstein-Uhlenbeck process is ergodic.
\end{lemma}
\begin{proof}
A stationary Gaussian process is ergodic if its spectral measure has no atom, see 
 \cite{Cornfeld-Fomin-Sinai, Samorodnitsky}. The spectral measure $F$  of a stationary Gaussian process is obtained from  
 Fourier transforming  its correlation function and 
$\rho(\lambda)=\int_\R e^{i \lambda x} dF(x)$.
According to   \cite{Cheridito-Kawaguchi-Maejima}:
 \begin{equation}\label{cor5-2}
\rho(s) =  \frac{ \Gamma(2H+1) \sin(\pi H)}{2 \pi} \int_{\R} e^{isx} \frac{\vert x \vert^{1-2H}}{1 + x^2} dx,
\end{equation}
so the spectral measure is absolutely continuous with respect to the Lebesgue measure with spectral density $s(x) = c \f { \vert x \vert^{1-2H}}{1+ x^2}$. 
\end{proof}
For $k \in \N $, we define the $\F_k$-adapted processes:
\begin{align*}
	I(k)&= \int_{k-1}^{k} U(y_s) ds= \Phi_U(k-1) -\Phi_U(k)  \\
	J(k)&= \int_{k-1}^{k}  V(y_s) ds=\Phi_V(k-1) -\Phi_V(k).
\end{align*}

	\begin{remark}\label{remark-mart-dif}
We note the following useful identities. For $k \in \N $
\begin{equs}\label{mart-dif2}
&\hat U(k) = I(k)+ \E [\hat U(k+1) \, |\, \F_k],  \\
&M_{k+1}-M_k=I(k)+\hat U(k+1)-\hat U(k),\label{mart-dif} \\
& \sum_{j=1}^k I(j) = \int_0^k U(y_r) dr =M_k-\hat U(k)+\hat U(1)-M_1.
\end{equs}
and similarly for $V$ and $N$.
\end{remark}

Henceforth in this section we set $L=L(\epsilon)=[\f t \epsilon]$.

\begin{lemma}\label{lemma2-area}
There exists a function $\err_1(\epsilon)$, which converges to zero in probability as $\epsilon\to 0$,   such that
\begin{equation}\label{area1} \begin{aligned}
&\epsilon \int_{0}^{\f t \epsilon} \int_0^s  U(y_s) V(y_r) dr ds=\epsilon\sum _{k=1} ^{L} I(k) \sum_{l=1}^{k-1} J(l)+t \int _0^1 \int_0^s \E \left( U(y_s)  V(y_r)\right) \,dr  ds
+\err_1(\epsilon)\end{aligned}
\end{equation}
\end{lemma}

\begin{proof} 
Let us divide the integration region  $0\le r\le s \le \f t {\epsilon }$ 
into several parts,
\begin{align*}
	\int_{0}^{L} \int_0^s  U(y_s) V(y_r) dr ds
+ \int_{L} ^{\f t \epsilon }  \int_0^s  U(y_s) V(y_r) dr ds.\end{align*}
The second term  is of order $o(\epsilon)$ since
 $\| \int_{L} ^{\f t \epsilon }  U(y_s) ds\|_{L^2(\Omega)}$ is bounded by stationarity   of $y_r$ and Theorem \ref{theorem-CLT}, see also \cite{Gehringer-Li-fOU}, furthermore,
 the term $\|\sqrt \epsilon \int_0^{\f t \epsilon}  V(y_r) dr\|_{L^2(\Omega)}$ is bounded by $ \f t {\sqrt{\epsilon}}$.
We compute for the remaining part,
\begin{align*}
	\int_{0}^{L} \int_0^s  U(y_s) V(y_r) dr ds
	 =&\sum _{k=1} ^{L} \int _{k-1}^{k} U(y_s) 
	\left(   \int_0^{k-1}  V(y_r)dr+ \int_{k-1}^s  V(y_r)dr  \right)   ds\\
	=&\sum _{k=1} ^L\int _{k-1}^{k} U(y_s) ds  \int_0^{k-1}  V(y_r)dr + \sum _{k=1} ^{L} \int _{ \{ k-1\le r\le s \le k\}} U(y_s)  V(y_r)dr  ds
	\\=&\sum _{k=1} ^{L} I(k) \sum_{l=1}^{k-1} J(l)+ \sum _{k=1} ^{L} \int _{ \{ k-1\le r\le s \le k\}} U(y_s)  V(y_r)dr  ds.
\end{align*}
The stochastic process  $Z_k = \int _{ \{ k-1\le r\le s \le k\}} U(y_s)  V(y_r)dr  ds$
 is shift invariant and the shift operator is ergodic with respect to the probability distribution on the path space generated by the fOU process,
 hence, by Birkhoff's ergodic theorem, 
$$ \f 1 {L} \sum_{k=1}^ {L} Z_k \stackrel{(\epsilon\to 0)} {\longrightarrow}  \E Z_1= \int _0^1 \int_0^s \E \left( U(y_s)  V(y_r)\right) dr  ds .$$
This completes the proof.
\end{proof}

\begin{lemma}\label{lemma-6.15}
The following converges in probability:
$$\lim_{\epsilon\to 0} \left(  \epsilon\sum_{k=1} ^{L} I(k) \sum_{l=1}^{k-1} J(l)-  \epsilon  \sum_{k=1} ^{L}   (M_{k+1}-M_k)N_k\right)
= t \;  \int_1^{\infty} \int_0^1 \E\left( U(y_s) V(y_r)\right) dr ds.$$
\end{lemma}
\begin{proof}

{\bf A.} 
Following \cite{roughflows} and using  the identities of Remark \ref{remark-mart-dif}, 
we obtain:
 \begin{align*}
 & \sum_{k=1} ^{L}\left( I(k) \sum_{l=1}^{k-1} J(l) -  \left(  M_{k+1} - M_k \right) N_k\right)  \\
 &= \sum_{k=1} ^{L} I(k)\,\left ( N_k -\hat V(k)+\hat V(1)-N_1 \right)-\left(I(k)+  \hat U ({k+1})-\hat U(k) \right) N_k\\
 &=  \sum_{k=1} ^{L} -I(k)\, \hat V(k) + \sum_{k=1} ^{L} I(k)(\hat V(1)-N_1)- \sum_{k=1} ^{L} (\hat U (k+1)-\hat U(k) ) N_k.
	  \end{align*}
Firstly, by the  shift invariance of the  summands below and Birkhoff's ergodic theorem we obtain
\begin{equation}\label{2nd-limit}
\begin{aligned}
-\, \epsilon \sum_{k=1}^{L } I(k) \hat V(k)&\longrightarrow (-  t)\, \E [I(1) \hat V(1)]=(-t)  \E \left( \int_0^1 U(y_r) dr\int_0^{\infty} V(y_s) ds \right).
 \end{aligned}
\end{equation}
Next, since $\hat V(1)-N_1=\E [\hat V(1) \, |\, \F_0]$, 
  $$\begin{aligned}
 \E \left| \epsilon  \sum_{k=1}^{L } I(k)(\hat V(1)-N_1) \right|^2
=&\E \left| \epsilon \int_0^{L } U(y_r) \;dr   \; \E [\hat V(1) \, |\, \F_0]\right|^2\\
\lesssim &\epsilon^2\, \E[\hat V(1)]^2  \int_0^{L }  \int_0^{L }  \E[ U(y_r)U(y_s)]\, ds \,dr,\end{aligned}$$
which by Lemma \ref {Integrals} and expanding into Hermite polynomials converges to $0$ as $\epsilon \to 0$.

{\bf B.} It remains to discuss the convergence of 
$$\epsilon \sum_{k=1}^{L }  (\hat U (k+1)-\hat U(k) ) N_k.$$ 
We change the order of summation to obtain the following decomposition
$$ \begin{aligned}
& \sum_{k=1}^L  (\hat U (k+1)-\hat U(k) ) N_k
\\=& \sum_{k=1}^L(\hat U(k+1) -\hat U(k)) \left[  \sum_{j=1}^{k-1}  (N_{j+1}-N_j) +N_1\right]\\
=&  \sum_ {j=1}^{L-1} (N_{j+1}-N_j)  \sum_{k=j+1}^L (\hat U(k+1) - \hat U(k)) 
 + \sum_{k=1}^L(\hat U(k+1) -\hat U(k))  N_1\\
 =&\ \sum_ {j=1}^{L-1} (N_{j+1}-N_j)   \hat U(L+1)  - \sum_{j=1}^{L-1} (N_{j+1}-N_j) \,\hat U(j+1)
 +\left (\hat U(L+1) - \hat U(1)\right)  N_1.
\end{aligned}$$
We may now apply Birkhoff's ergodic theorem to  the first term, taking $\epsilon\to 0$,
$$\lim_{\epsilon\to 0} - \epsilon\ \sum_ {j=1}^{L-1} (N_{j+1}-N_j)   \hat U(L +1) =0,\quad $$ 
in probability.
By the same ergodic theorem, the second term converges to 
$  -t\, \E\left( \hat U(2) (N_2-N_1) \right) $
in probability.
By Proposition \ref {prop-integrability}, $\hat U(j)$ is bounded in $L^2(\Omega)$, hence, for the third term we obtain,
$$\epsilon \left|\left (\hat U(L +1) - \hat U(1)\right)  N_1\right|_{L^2(\Omega)} \lesssim \epsilon. $$ 
Overall we end up with \begin{equation}
\lim_{\epsilon \to 0} (-\epsilon) \sum_{k=1}^L  (\hat U _{k+1}-\hat U(k) ) N_k =  t\; \E\left( \hat U(2) (N_2-N_1) \right),
\end{equation}
where the convergence is in probability, hence,  \begin{equation}\label{diff}
\begin{split}
&\lim_{\epsilon\to 0} 
\left(  \epsilon \sum_{k=1}^{L} \sum_{l=0}^{k} I(k) J(l)-  \epsilon \sum_{k=1}^{L}\left(  M_{k+1} - M_k \right) N_k\right)
=  t \,\E \left[  \hat U(2) (N_2-N_1)  -  I(1) \hat V(1) \right].
\end{split}
\end{equation}
{\bf C.} We look for a better expression of the limit in (\ref{diff}). Firstly by Corollary \ref{cor-integrability},
we have $$\begin{aligned} (N_2-N_1) &=  \hat V(2) -\E (\hat V(2)|\F_1)=\hat V(2) -\E (\hat V(2)|\F_1) 
\\&=\hat V(2)-\hat V(1) + \int_0^1 V(y_s)\, ds.\end{aligned}$$ 
Using this and  $I(1)=  \int_0^1 U(y_s)ds=\hat U (1) -   \E [\hat U (2) | \F_1 ]$,  we compute 
\begin{align*}
&\E \left(  \hat U(2) (N_2-N_1)  -  I(1) \hat V(1) \right)\\
&=\int_1^{\infty} \int_0^1\E \left(  U(y_s) V(y_r)  \right) dr ds + \E \left(  \hat U (2) \left(\hat V(2) -\hat V(1)\right) - \left(\hat U (1) -   \E [\hat U (2) | \F_1 ]\right) \, \hat V(1)  \right).
\end{align*}
Since $\hat V(1)$ is $\F_1$ measurable, 
\begin{align*}
\E \left(  \hat U (2) \left(\hat V(2) -\hat V(1)\right) - \left(\hat U (1) -   \E [\hat U (2) | \F_1 ]\right) \, \hat V(1)  \right),
\end{align*}
vanishes by the shift covariance of $\hat U(k) \hat V(k)$. This concludes the proof of the lemma.
\end{proof}

 {\bf Proof of  Proposition \ref{prop-area}.} 
	
Combining  (\ref{area1}),  Lemma \ref{lemma-6.15} and Lemma \ref{lemma-6.13},
we have
\begin{align*}
& \epsilon \int_0^{\f t \epsilon} \int_0^s U(y_s) V(y_r) dr ds\\ 
=& \epsilon\sum _{k=1} ^{L} I(k) \sum_{l=1}^{k-1} J(l)+t \int _0^1 \int_0^s \E \left( U(y_s)  V(y_r)\right) \,dr  ds + \err_1(\epsilon) \\
=& \epsilon \sum_{k=1} ^{L}   (M_{k+1}-M_k)N_k 
+t   \int_1^{\infty} \int_0^1 \E\left( U(y_s) V(y_r)\right) dr ds 
\\&+ t \int _0^1 \int_0^s \E \left( U(y_s)  V(y_r) \right)dr  ds
+\err_1(\epsilon) + \err_2(\epsilon)\\
=& \epsilon \sum_{k=1} ^{L}   (M_{k+1}-M_k)N_k 
+t   \int_0^{\infty}\E \left(U(y_v) V(y_0) \right)  du
+\err_1(\epsilon) + \err_2(\epsilon) ,
\end{align*}
where we used stationarity, $\E (U(y_s) V(y_r) ) =\E (U(y_{s-r})V(y_0))$, and the change of variables $u= {s+r} $, $v= { s-r}$, leading to the identity,
$$
\left( \int_0^1 \! \!\int_0^s  + \int_{1}^{\infty}\! \!\int_0^1 \right) \E \left(U(y_s) V(y_r) \right) dr ds 
=-\f 12\int_0^{\infty}  \! \! \int_u^{u-2}\E \left(U(y_v) V(y_0) \right) du dv
=\int_0^{\infty}\E \left(U(y_v) V(y_0) \right)  dv.$$
This completes the proof of Proposition \ref{prop-area}.

\begin{proposition}\label{lemma-6.13}
Let $X^{W,\epsilon}=(X^{1,\epsilon} , \dots, X^{n,\epsilon})$, then
$$ (X^{W,\epsilon} , \XX^{W,\epsilon}) = (X^{W,\epsilon} , \int_0^t X^{W,\epsilon} dX^{W,\epsilon}) \to  (X^W,\int_0^t X^W  dX^W + t A),$$ 
jointly in finite dimensional distributions, where the integration is understood in the It\^o sense and $A$ is as in Theorem \ref{theorem-lifted-CLT}.
\end{proposition}
\begin{proof}
	For each $X^k$ we first define the  martingales
	$M^{k} $ as in Corollary \ref{cor-integrability}, for  $U=G_k$. 
	Then, we define the C\'adl\'ag martingales as follows  
	$$M^{k,\epsilon}_{t} = {\sqrt{\epsilon}}  M^k_{[\f t \epsilon]}.$$
	Using the identity (\ref{mart-dif}) we obtain,
	$$\begin{aligned}
	M^{k,\epsilon}_{t}  &=\sqrt \epsilon \sum_{q=1}^{[\f t \epsilon]} (M^k_{q+1} -M^k_q) +\sqrt \epsilon \,M^k_1 =\sqrt \epsilon \int_0^{[\f t \epsilon]} G_k(y_r) ds
	+\sqrt \epsilon \hat G_k\left(\left[\f t \epsilon\right]\right)-\sqrt \epsilon \hat G_k(1) +\sqrt \epsilon M^k_1.
	\end{aligned}$$
	Since $\hat G_k$ is $L^2(\Omega)$ bounded, the joint convergence  $(M^{1,\epsilon}_{t}, \dots, M^{n,\epsilon}_{t}) \to X^W  $ in finite dimensional distributions follows from Theorem \ref{theorem-CLT}.  Next by Proposition \ref{prop-area},
	$$ \left( \int_0^t X^W_s dX^W_s \right)^{i,j}= \epsilon \sum _{q=1} ^{[\f t \epsilon]} (M^j_{q+1}-M^j_q)M^{i}_q + tA^{i,j} + \err(\epsilon)
	= \int_0^{t} M^{i,\epsilon}_{s} dM^{j,\epsilon}_s + tA^{i,j} + \err(\epsilon),$$
	where the integration is understood in the It\^o sense and $\err(\epsilon) \to 0 $ in probability.
	The joint convergence of $( M^{k,\epsilon}, \int_0^{t} M^{i,\epsilon}_{s} dM^{j,\epsilon}_s )_{i,j,k \leq n}$ in finite dimensional distributions follows, as for each $k$,
	$\E \left( M^{k,\epsilon}_t \right)^2 \lesssim t + o(\epsilon)$, by an application of Theorem 2.2 in \cite{Kurtz-Protter}, which states that given a sequence of jointly convergent martingales bounded in $L^2(\Omega)$, then these martingales also converge jointly with their It\^o integrals. This concludes the proof. \end{proof}

We summarise this section with the following more general statement, which follows from the proofs above:
\begin{remark}
Let $y_t$ be a stationary and ergodic stochastic process with stationary measure $\pi$, $G_k \in L^2(\pi)$, $X^{k,\epsilon}_t= \sqrt{\epsilon} \int_0^{\f t \epsilon} G_k(y_s) ds$ such that $X^\epsilon=(X^{1,\epsilon}, \dots, X^{n,\epsilon})$ converges, as $\epsilon \to 0$, to a Wiener process $X$.
Suppose that 
$(\int_{k-1}^\infty \E(G_j(y_r)|\F_k) dr, k\ge 1)$ is $L^2(\Omega)$ bounded. Then,  $\X^{\epsilon}=(X^\epsilon, \XX^\epsilon)$, the canonical rough path lift of $(\X^\epsilon)$, converges to $(X, \XX+(t-s)A)$, the Stratonovich lift of $X$.
(By this we mean that $\XX^{i,j}_{0,t}=\int_0^t X^i_s dX^j_s$, where the integral is understood in the It\^o sense and $A$ denotes the corresponding Stratonovich correction.)
\end{remark}

\subsection{Proof of Theorem  \ref{theorem-lifted-CLT}}
\label{iterated-CLT}
Now we are ready to conclude the convergence  of $\X^{\epsilon}$ weakly in the rough path topology. We assume that $G_k \in L^{p_k}(\mu)$ satisfy the integrability conditions specified in Assumption \ref{assumption-multi-scale}. Fix $H \in (0,1) \setminus \{ \f  1 2\}$ and a final time $T$.  Let, for $ t \in [0,T]$,
\begin{align*}
X^\epsilon_t&:=\left(\alpha_1(\epsilon) \int_0^{t} G_1(y^{\epsilon}_s)ds, \,\dots, \alpha_N(\epsilon) \int_0^{t} G_N(y^{\epsilon}_s)ds\right)\\
A^{i,j} &=  \left\{	\begin{array}{cl} 
\int_0^{\infty} \E\left( G_i(y_s) G_j(y_0) \right) ds,  &\hbox{ if }i,j \leq n ,\\
0, &\hbox{ otherwise.}
\end{array} \right.\end{align*}
By Theorem \ref{theorem-CLT}  $  X^{\epsilon}$ has a limit which we denote by $X$.
Set further,  $$ {\begin{split} \XX^{i,j,\epsilon}_{u,t}&= \alpha_i(\epsilon) \alpha_j(\epsilon) \int_u^{t} \int_u^s G_i(y^{\epsilon}_s) G_j(y^{\epsilon}_r) dr ds, \\
		\XX^{i,j}_{u,t} & =  \int_0^t X_{u,s}^i dX_s^j, \end{split}}$$
	where the second  integral is to be understood in the It\^o-sense if two Wiener processes appear, and in the Young sense otherwise.
We will prove below that,  as $\epsilon\to 0$,
$$ \X^{\epsilon} = \left(X^\epsilon, \XX^{\epsilon}  \right) \to \X = \left(X,\XX + A(t-s) \right),$$ 
	weakly in $\FC^{\gamma}([0,T],\R^N)$ for $\gamma \in (\f 1 3, \f 1 2 - \f {1} {\min_{k\leq n } p_k}).$

\begin{proof}
Firstly, we recall that by Theorem \ref{theorem-CLT} the basic processes converge jointly and the limiting Wiener process, $X^W$, is independent of the limiting Hermite process, $X^Z$. In Proposition \ref{lemma-6.13} we showed that the high Hermite rank components, $X^{W,\epsilon}$, converge jointly with their iterated integrals $\int_0^t X^{W,\epsilon} dX^{W,\epsilon}$. In Lemma \ref{young-lift} we proved that the first order processes  together with the lifts for which $ i \vee j >n$ converge jointly and in particular these lifts are continuous functionals of $(X^{W,\epsilon},X^{Z,\epsilon})$. 
By the continuous dependence of $ \int_0^t X^{i,\epsilon} dX^{j,\epsilon}$, where $i \vee j >n$, we may leave out these iterated integrals,  it is sufficient to show joint convergence of $(X^{W,\epsilon},\int_0^t X^{W,\epsilon} dX^{W,\epsilon},X^{Z,\epsilon})$. This vector is tight in $\FC^{\gamma} \times \C^{\f 1 2}$ by Theorem \ref{theorem-CLT}, Proposition \ref{tightness-lemma}, and application of Kolmogorov's Theorem (see also Lemma \ref{tightness-second-order} below), hence we may chose a converging subsequence. We now want to identify the limiting distribution. As we have already established convergence of the marginals $(X^{W,\epsilon},\int_0^t X^{W,\epsilon} dX^{W,\epsilon})$ and $X^{Z,\epsilon}$, by independence of $X^W$ and $X^Z$ their limiting distribution is just given by the product measure between $(X^W, \int_0^t X^W dX^W)$ and $X^Z$. 
The choice of subsequence was arbitrary, hence each subsequence has the same limit and the whole sequence converges. This concludes the proof.
\end{proof}

 \subsection{Proof of the conditional integrability of fOU}
\label{prove-lemma-int}
The aim of this section is to prove that  $\sup_{k}\int_{k-1}^{\infty} \int_{k-1}^{\infty}   \E \left(    {\overline{y}^k_s} {\overline{y}^k_r}  \right)^m  dr ds $ is indeed finite, for which we restrict ourselves to the case $H \in (0,1) \setminus \{ \f  1 2\}$ and  $H^*(m)<0$.
We first compute the  conditional expectations of  $\E(G(y_t) | \mathcal{F}_k)$ where $G \in L^2(\mu)$.

In \cite{Additive} it was show that a fractional Brownian motion has 
a locally independent decomposition: for any $k<t$, $B_t-B_k=\tilde B_t^k+\bar B_t^k$, in which the smooth process $\bar B_t^k$ is adapted to $\F_k$ and the rough part $\tilde B_t^k$ is independent of $\F_k$. This  is given by a Mandelbrot Van-Ness representation, indeed, up to a multiplicative factor,
$$\overline{B}^k_t =  \int_{-\infty}^{k} (t-r)^{H - \f 1 2} - (k-r)^{H- \f 1 2} dW_r ,
\qquad  \tilde{B}^k_t=\int_k^{t} (t-r)^{H- \f 1 2} dW_r.$$
Furthermore, it was shown that  the filtration generated by the fractional Brownian motion is the same as the one generated by the two-sided Wiener process $W_t$. 
We now prove such a decomposition for the fractional Ornstein-Uhlenbeck process.
  
\begin{lemma}
Let $\F_s$ be the filtration generated by $B^H$.  For $k<t$ define $$ \bar y^k_t=\left( \int_{-\infty}^k e^{-(t-r)} dB_r +\int_{k}^{t} e^{-(t-r)} d\overline{B}^k_r \right), \qquad 
 \tilde{y}^k_t= \int_{k}^{t} e^{-(t-r)} d\tilde{B}^k_r.$$
 Then, $\tilde y_t^k$ is independent of $\F_k$, $\bar y_t^k\in \F_k$, both are Gaussian random variables and $y_t=\bar y_t^k+\tilde y_t^k$.
 (In case $t \leq k$ we set $\overline{y}^k_t = y_t$.)
\end{lemma}
\begin{proof}
Splitting the integral and using,  $B_r-B_k=\tilde B_r^k+\bar B_r^k$, we obtain,
\begin{align*}
y_t &= \int_{-\infty}^{t} e^{-(t-r)} dB_r
= \int_{-\infty}^k e^{-(t-r)} dB_r +   \int_{k}^{t} e^{-(t-r)} d( B_r - B_k)\\
&=  \left( \int_{-\infty}^k e^{-(t-r)} dB_r +\int_{k}^{t} e^{-(t-r)} d\overline{B}^k_r \right)+  \int_{k}^{t} e^{-(t-r)} 
 d\tilde{B}^k_r\\
&= \overline{y}^k_t + \tilde{y}^k_t,
\end{align*}
where the first term $\overline{y}^k_t$ is $\mathcal{F}_k$ measurable and $\tilde{y}^k_t$ is independent of  $\mathcal{F}_k$.
\end{proof}

\begin{lemma}\label{lemma-integral}
	Let $\tau > -1$.
	For any $k<s$, 
	$$ \int_k^s e^{-(s-v)}  (v-k)^{\tau} \,dv  \lesssim  1 \wedge (s-k)^{\tau}.$$ 
	The $ \lesssim $ sign indicate that the constant on the right hand side is  independent of $k$ and $s$ .
\end{lemma}
\begin{proof}
	We may assume that $s>4k+4$, otherwise the integral is finite as the exponential term can be estimated by $1$ and as $\tau > -1$, the singularity is integrable.
	Splitting the integral  into two regions  $\int_k^{k+1}+\int_{k+1}^s$, we have  $\int_k^{k+1}   e^{-(s-v)}  (v-k)^{\tau} \,dv \lesssim   e^{-(s-k-1)}$ and furthermore  using integration by parts,
	$$\begin{aligned}
	&\int_{k+1}^{s}  e^{-(s-v)}  (v-k)^{\tau} \,dv\\
	&=  (s -k )^{\tau} -e^{-(s-k-1)}   -\tau \left( \int_{k+1} ^{\f s 2} +\int_{\f s 2}^s \right) e^{-(s-v)} (v-k)^{\tau-1}  dv\\
	&\lesssim (s-k)^{\tau} -e^{-(s-k)} -e^{-\f s 2} (v-k)^{\tau} |_{k+1}^ {\f s 2} -(v-k)^{\tau} |_{\f s 2}^s\\
	&\lesssim  (s-k)^{\tau}+e^{-\f s 2} +(\f s 2 -k)^{\tau}
	\lesssim (s-k)^{\tau}.
	\end{aligned}$$
	This gives the required estimate.
\end{proof}

\begin{lemma}\label{variance-decay-k}
	For $ t \geq k-1$ the following estimate holds,
	$\Vert \bar{y}^k_t \Vert_{L^2(\Omega)} \lesssim  1 \wedge \vert t-k \vert^{H-1} $.
\end{lemma}
\begin{proof}
	Firstly, as $ \Vert y_t \Vert_{L^2(\Omega)} =1$ we also obtain $\Vert \bar{y}^k_t \Vert_{L^2(\Omega)} \leq 1 $. Thus, it is only left to consider the behaviour when $t$ becomes large. 	
	Using the above Lemma we obtain
	\begin{align*}
	\Vert y^k_t \Vert_{L^2(\Omega)} &= \Vert e^{-(t-k)} y_k + \int_{k}^t e^{-(t-s)} d\bar{B}^k_s \Vert_{L^2(\Omega)} \leq e^{-(t-k)} \Vert y_k \Vert_{L^2(\Omega)} + \int_{k}^t e^{-(t-s)} \Vert \dot{\bar{B}}^k_s \Vert_{L^2(\Omega)} ds\\
	&\leq  e^{-(t-k)} + \int_{k}^t e^{-(t-s)} \vert s-k \vert^{H-1} ds \lesssim \vert t-k \vert^{H-1}.
	\end{align*}
\end{proof}

\begin{proposition}\label{integrable-lemma}
	Given $H \in (0,1) \setminus \{ \f  1 2\}$ and suppose that  $H^*(m)<0$. Then,
	$$\sup_k  \int_{k-1}^{\infty} \int_{k-1}^{\infty}  \left( \E \left(    {\overline{y}^k_s} {\overline{y}^k_t}  \right) \right)^m \, dt \,ds < \infty.$$ 
\end{proposition}
\begin{proof}
	As
	\begin{align*} 
	\int_{k-1}^{\infty} \int_{k-1}^{\infty}  \left( \E \left(    {\overline{y}^k_s} {\overline{y}^k_t}  \right) \right)^m \, dt \,ds &\leq \int_{k-1}^{\infty} \int_{k-1}^{\infty}  \left( \Vert {\overline{y}^k_s} \Vert_{L^2(\Omega)} \Vert {\overline{y}^k_t}  \Vert_{L^2(\Omega)} \right)^m  \, dt \,ds\\
	&= \left( \int_{k-1}^{\infty}  \Vert  {\overline{y}^k_t}  \Vert_{L^2(\Omega)}^m   dt \right)^2,
	\end{align*}
	it is sufficient to show finiteness of  $\int_{k-1}^{\infty}  \left( \E    {\left(\overline{y}^k_s\right)}^2 \right)^{\f m 2}  dt $. By Lemma \ref{variance-decay-k} we obtain,
	\begin{align*}
	\int_{k-1}^{\infty}  \left( \E  (  {\left(\overline{y}^k_s\right)}^2) \right)^{\f m 2}  dt \lesssim \int_{k-1}^{\infty} 1 \wedge \vert t-k \vert^{q(H-1)}  dt.
	\end{align*}
	This expression is finite if $m(H-1) < -1$. As  $H^*(m) = m(H-1)+1 <0$ this concludes the proof.
\end{proof}

\section{ Multi-scale homogenisation theorem }\label{conclusion}
For Hermite polynomials we have the hypercontractivity estimate:
 $$\|H_k\|_{L^{2q}(\mu)}\le  (2q-1)^{\f k2 }\sqrt{ \E (H_k)^2} = (2q-1)^{\f k2 }\sqrt{k!}.$$
Consequently,  if an $L^2(\mu)$  function $G=\sum_{l=0}^\infty  c_l H_l$ satisfies the fast chaos decay condition with parameter~$q$,
 then, 
$\|G\|_{L^{2q}(\mu)} \le \sum_{l=0}^\infty  \vert c_l \vert \|H_l\|_q<\infty$.
 We used $\sum_{l=0}^\infty  {|c_l|}\; \sqrt{l!} \;(2q-1)^{\f l2}<\infty$. 
Thus,  $G\in L^{2q}(\mu)$.  Observe that $\f 12 -\f 1{2q} >\f 13$, a condition needed for the convergence in $\FC^\gamma$,   is equivalent to the condition $q>3$. 
Also, if $G$ satisfies the decay condition with $q > 1$, then $G$ has a continuous representation.
 Indeed, we have
$$\left|e^{-\f {x^2} 2} H_k(x) \right|_\infty \le 1.0865\sqrt{k!},$$
 see  \cite[pp787]{Abramowitz-Stegan}, the polynomials  in \cite{Abramowitz-Stegan} are orthogonal with respect to $e^{-x^2} dx$ and one should take care with the convention.
 Thus the power series  $e^{-\f {x^2} 2}\sum_{l=0}^\infty   c_l H_l$ converges uniformly in $x$, the limit $G$ is
continuous. 
\begin{remark}
If $G$ satisfies the fast chaos decay condition with parameter $q >1$, then $G$ has a representation in $L^{2q} \cap \C$, with which we will work from here on.
\end{remark}

We reformulate the main theorem here where it is proved.
\begin{theorem}\label{thm-homo-2}
\label{proof} Let $H \in  (0,1) \setminus \{ \f  1 2\}$, $f_k\in \C_b^3( \R^d , \R^d)$, and $G_k$  satisfy Assumption \ref{assumption-multi-scale}. Set $f=(f_1, \dots, f_N)$, then, the following statements hold.
\begin{enumerate}
\item The solutions  $x_t^\epsilon$ of (\ref{multi-scale}) 
converge weakly  in  $\C^\gamma$ on any finite interval and for any $\gamma \in (\f 1 3 ,\f 1 2 - \f 1 {\min_{k \leq n } p_k})$.
\item 
The limit  solves  the rough differential equation
\begin{equation}
\label{effective-rde}
 d {x_t} =f(x_t) d \X_t \quad  x_0=x_0.
 \end{equation}
 Here $ \X=(X,\XX_{s,t}+(t-s)A)$ is a rough path over $\R^N$, as specified in Theorem \ref{theorem-lifted-CLT}.
 \item  Equation (\ref{effective-rde}) is equivalent to the  stochastic  equation below:
 $$ d{x}_t =\sum_{k=1}^n f_k(x_t) \circ d X^k_t+\sum_{l=n+1}^N f_l(x_t) d X^l_t, \quad x_0=x_0,$$
where $(X^1, \dots, X^n)$  and $(X^{n+1}, \dots, X^N)$  are independent, the $\circ$ denotes  Stratonovich integral, and the other integrals  are Young integral. 
\end{enumerate}
\end{theorem}
\begin{proof}
We want to formulate our slow/fast random differential equation  as a family of rough differential equations such that the drivers converge in the rough path topology. Using the continuity of the solution map, we obtain weak convergence of the solutions to  a rough differential equation. Results in rough path then relate this rough differential equations to usual Stratonovich/Young equations, this is explained in \S\ref{sec:explain} where we introduce the notations from rough differential equations, see also \cite{Friz-Victoir,Friz-Hairer,Lyons-Caruana-Levy}.
We define  $F: \R^d \to {\mathbb L} (\R^N, \R^d)$  as follows:
$$F (x) (u_1, \dots, u_N)= \sum_{k=1}^N    u_kf_k(x).$$
If we further set $$G^\epsilon=\Big(\alpha _1(\epsilon) G_1, \dots, \alpha_N(\epsilon) G_N\Big),$$
we may then write our slow equation as
$$\dot x_t^\epsilon =F(x_t^\epsilon) G^\epsilon(y_t^\epsilon).$$
For the rough path $\X^{\epsilon}=(X^{\epsilon},\XX^{\epsilon})$ defined by (\ref{x-epsilon},\ref{lifted-x-epsilon}).
 we may rewrite equation (\ref{multi-scale})  as  a rough differential equation with respect to $\X^\epsilon$:
$$d x_t^\epsilon =F(x_t^\epsilon) d \X^\epsilon(t).$$
with covariance as specified in Theorem \ref{theorem-CLT} and Theorem \ref{theorem-lifted-CLT}.

 By Theorem \ref{theorem-lifted-CLT},
$\X^{\epsilon}$ converges to $\X=(X, \XX + (t-s)A)$ in $\FC^{\gamma}$ where $\gamma \in (\f 1 3 , \f 1 2 - \f {1} {\min_{k \leq n } p_k })$ on every finite interval.
Since  $\gamma>\f 13$ by Assumption \ref{assumption-multi-scale}, 
we may apply the continuity theorem for rough differential equations, Theorem \ref{cty-rough},
to conclude that the solutions converge to the solutions of the rough differential equation
$$\dot x_t=F(x_t) d\X_t.$$
Since $F$ belongs to $\C_b^3$, this is well posed as a rough differential equation. 
We completed the proof for the convergence.
To show the independence of $X^k$ for $k \leq n$ from the other processes we observe that by Assumption \ref{assumption-multi-scale} the terms of $\XX^{i,j}$ for which at least $i>n$ or $j>n$ do not contribute in the limit, hence, we conclude the proof by Theorem \ref{theorem-lifted-CLT}.
\end{proof}

\section{Appendix}
The purpose of the appendix is to explain the notation we used from rough path theory.
We include the theorems needed for proving the tightness theorem and the homogenisation theorem.
Finally we explain how to interpret the effective rough differential equations (\ref{limit-eq}) with It\^o integrals and Young integrals, and hope this self-contained material will be useful for those not familiar with the rough path theory.

\subsection{Some rough path theory} \label{rough-path}
If $X$ and $Y$ are H\"older continuous functions on $[0,T]$ with exponent $\alpha$ and $\beta$ respectively, such that $\alpha + \beta >1$, the Young integration theory enables us to define $\int_0^T Y dX$ via limits of Riemann sums $\sum_{[u,v]\in \mathcal P} Y_u(X_v-X_u)$, where $\CP$ denotes a partition of $[0,T]$. Furthermore $(X,Y)\mapsto \int_0^T Y dX$ is a continuous map. Thus, for $X\in \C^{\f 12+}$, one can make sense of a solution $Y$ to the Young integral equation $dY_s=f(Y_s) dX_s$, given enough regularity on $f$. If $f\in \C_b^2$,  the solution is  continuous with respect to both the driver $X$ and  the initial data, see \cite{Young36}.
In the case of $X$ having H\"older continuity less or equal to $\f 12$,   this fails and one cannot define a pathwise integration for $\int X dX$ by the above Riemann sum anymore.  
Rough path theory provides us with a machinery to treat less regular functions by enhancing the process with a second order process, giving a better local approximation, which then can be used to enhance the Riemann sum and show it  converges. If $X_s$ is a Brownian motion and taking a dyadic approximation, then, the usual Riemann sum converges in probability to the It\^o integral. The enhanced Riemann sum, however, provides a better approximation and defines a pathwise integral agreeing with the It\^o integral provided the integrand belongs to both domains of integration. Their domains of integration
are quite different, the first uses an additional adaptedness condition and requires arguably less regularity than the second.
We restrict ourselves to the case where  $X_t$ is a continuous path over $[0,T]$, which takes values in $\R^d$. 
A rough path of regularity $ \alpha \in (\f 1 3 , \f 1 2)$,
is a pair of process $\X=(X_t, \XX_{s,t})$ where $(\XX_{s,t}) \in \R^{d \times d} $ is a two parameter stochastic processes
satisfying the following algebraic conditions: for $0\le s<u<t\le T$, 
$$\XX_{s,t}-\XX_{s,u}-\XX_{u,t}=X_{s,u} \otimes X_{u,t}, \qquad \qquad  \hbox{ (Chen's relation)} $$
where $X_{s,t}=X_t-X_s$, and $ (X_{s,u} \otimes X_{u,t})^{i,j}  = X^i_{s,u}  X^j_{u,t}$ as well as the following analytic conditions,
\begin{equation}\label{geo}
\Vert X_{s,t} \Vert \lesssim |t-s|^\alpha,  \qquad \Vert\XX_{s,t}\Vert \lesssim |t-s|^{2\alpha}.
\end{equation}
The set of such paths will be denoted by $\FC^{\alpha}([0,T]; \R^d)$. The so called second order process $\XX_{s,t}$ can be viewed as a possible candidate for the iterated integral $\int_s^t X_{s,u} dX_u$. 
\begin{remark}
Using Chen's relation for $s=0$ one obtains
$$ \XX_{u,t}= \XX_{0,t} - \XX_{0,u} - X_{0,u} \otimes X_{u,t},$$
thus one can reconstruct $\XX$ by knowing the path $t \to (X_{0,t}, \XX_{0,t})$.
\end{remark}
Given a  path $X$, which is regular enough to define its iterated integral, for example $X \in \C^1([0,T];\R^d)$, we define its natural rough path lift to be given by
$$\XX_{s,t}:=\int_s^t X_{s,u} dX_u.$$
It is now an easy exercise to verify that $\X = (X,\XX)$ satisfies the algebraic and analytic conditions (depending on the regularity of $X$), by which we mean Chen's relation and (\ref{geo}). Note that given any  function $F\in \C^{2\alpha}(\R^{d \times d})$, setting
$\tilde\XX_{s,t}=\XX_{s,t}+F_t-F_s$, $\tilde \XX$ would also be a possible choice for the rough path lift. 
Given two rough paths $\X$ and $ \Y$  we may define , for  $\alpha \in (\f 1 3, \f 1 2)$,  the  distance 
\begin{equation}\label{rough-distance}
\rho_\alpha(\X, \Y)=\sup_{s\not =t} \f {\Vert X_{s,t} -Y_{s,t} \Vert } {|t-s|^\alpha} 
+\sup_{s\not =t} \f {\Vert\XX_{s,t} -\YY_{s,t}\Vert }  {|t-s|^{2\alpha}} .
\end{equation}
This defines a complete metric on $\FC^{\alpha}([0,T]; \R^d)$, this is called  the inhomogenous $\alpha$-H\"older rough path metric.
We are also going to make use of the norm like object 
\begin{equation}
\Vert \X \Vert_{\alpha} = \sup_{s \not = t \in [0,T]} \f {\Vert X_{s,t}\Vert} {\vert t-s \vert^{\alpha} } +  \sup_{s \not = t \in [0,T]} \f {\Vert \mathbb{X}_{s,t}  \Vert^{\f 1 2}} {\vert t-s \vert^{\alpha}},
\end{equation}
where we denote for any two parameter process $\XX$ a semi-norm:
$$\|\XX\|_{2\alpha} := \sup_{s \not = t \in [0,T]} \f {\Vert \XX_{s,t}  \Vert} {\vert t-s \vert^{2\alpha}}.
$$

Given a path $X$, as the second order process $\XX$ takes the role of an iterated integral, another sensible conditions to impose  is the chain rule (or  integration by parts formulae) leading to the following definition.
\begin{definition}
	A rough path $\X$ satisfying the following condition,
	\begin{equation}
	Sym(\XX_{s,t})^{i,j} = \f{1} {2} \left( \XX^{i,j} + \XX^{j,i} \right)= \f 1 2 X_{s,t}^i \otimes X_{s,t}^j
	\end{equation} 
	is called a geometric rough path. The space of all of geometric rough paths of regularity $\alpha$ is denoted by $\FC^{\alpha}_g([0,T];\R^d)$ and forms a closed subspace of $\FC^{\alpha}([0,T];\R^d)$.
\end{definition} 
Furthermore,  one can show that if a sequence of $\C^1([0,T],\R^d)$ paths $X^n$ converges in the rough path metric to $\X$, then $\X$ is a geometric rough path. 
To obtain a geometric rough path from a Wiener process,  as $\int_0^t W_s \circ dW_s= \f {W_t^2} {2}$, one has to enhance it with its Stratonovich integral,
$\WW_{s,t}=\int_s^t (W_r-W_s)\circ dW_r$, up to an antisymmetric part.

Given a rough path $\X \in \FC^\alpha([0,T],\R^d)$, we may define the integral $ \int_0^TY d\X$ for suitable paths $Y \in \C^{\alpha}([0,T],\mathbb{L}(\R^d,\R^m))$, which admit a Gubinelli derivative $Y'\in \C^{\alpha}([0,T],\mathbb{L}(\R^{d \times d},\R^m))$ with respect to $\X$,  meaning
$Y_{s,t}=Y_s' X_{s,t}+R_{s,t},$
where the two parameter function $R$ satisfies $\|R\|_{2\alpha}< \infty$.  The pair $\Y:=(Y, Y')$ is said to be a
controlled rough path, their collection is denoted by $\D^{2\alpha}_X$. The remainder term for the case $Y=f(X)$ with $f$ smooth is
the remainder term in the Taylor expansion.
This is done by showing that the enhanced Riemann sums
$\sum_{[s,t]\in \CP} Y_{s} X_{s,t} + Y'_{s} \XX_{s, t},$
converge as the partition size is going to zero, and the limit is defined to be $\int \Y \,d\X$. Given  $Y \in \D_X^{2\alpha}$, then $(\int \Y\,d\X, Y)\in \D_X^{2\alpha}$, and the map $(\X,  \Y) \mapsto (\int \Y\, d\X, Y)$  is continuous with respect to $\X \in  \FC^{\alpha} $ and $Y\in \D_X^{2\alpha}$.

With this theory of integration one can  study the equation,
$$ dY = f(Y) d\X.$$
However, unlike in the theory of stochastic differential equations one now has continuous dependence on the noise $\X$. We now state the precise theorem for our application, see also \cite{Lyons94}.
\begin{theorem}{\cite{Friz-Hairer}}\label{cty-rough}
	Let $Y_0 \in \R^m, \beta \in (\f1 3, 1), \, f \in \C^3_b(\R^m, {\mathbb L} (\R^d, \R^m)) $ and $\X \in \FC^{\beta}([0,T],\R^d)$. Then, the differential equation
	\begin{equation}\label{example-sde}
	Y_t = Y_0 + \int_0^t f(Y_s) d\emph{X}_s 
	\end{equation}
	has a unique solution which belongs to $\mathcal{C}^{\beta}$. Furthermore, the solution map  $\Phi_f: ~\R^d\times \FC^{\beta}([0,T], \R^d)
	\to  \D_{X}^{2\beta}([0,T],\R^m)$, where the first component is the initial condition and the second component the driver, is continuous.
\end{theorem}
As continuous maps preserve weak convergence to show weak convergence of solutions to rough differential equations
$$ dY^{\epsilon} = f(Y^{\epsilon}) d\X^{\epsilon},$$
it is enough to establish weak convergence of the rough paths $\X^{\epsilon}$ in the topology defined by the rough metric. 
Obtaining convergence in  this topology follows the  convergence of the finite dimensional distributions of the rough paths $\X^\epsilon$ plus tightness in the space of rough paths with respect to that topology. 
\subsubsection{Tightness of rough paths}
\label{pre-tightness}
The following lemma  can be obtained via an Arzela-Ascoli argument, for details see \cite{Friz-Hairer,Friz-Victoir}.
\begin{lemma}\label{compact}
	Let $0$ denote the rough path obtained from the $0$ path enhanced with a $0$ second order process, then, for $\gamma > \gamma'> \f 1 3$, the sets $\{\X \in \FC^{\gamma'} : \rho_{\gamma}(\emph{X},0) < R, \X(0)=0 \}$ are compact in $\FC^{\gamma'}$.
\end{lemma} 

\begin{lemma}\label{moment-conditions}
	Let $ \theta \in (0,1)$, $\gamma \in (\f 1 3, \theta - \f 1 {p})$ and $\X^{\epsilon}=(X^\epsilon, \XX^\epsilon)$ such that 
	\begin{align*}
	\Vert {X}^{\epsilon}_{s,t} \Vert_{L^{p}(\Omega)} \lesssim \vert t-s \vert^{\theta}, \qquad 
	\Vert {\XX}^{\epsilon}_{s,t} \Vert_{L^{\f p 2}(\Omega)} \lesssim \vert t-s \vert^{ 2 \theta},
	\end{align*}
	then, 
	$$\sup_{\epsilon \in (0,1]} \E \left(\Vert\X^{\epsilon}\Vert_{\gamma} \right)^{p} < \infty.$$ 
\end{lemma}
\begin{proof}
	The proof is based on a Besov-H\"older embedding, for details we refer to \cite{Friz-Victoir,Chevyrev-Friz-Korepanov-Melbourne-Zhang}.
\end{proof}
\begin{lemma}\label{tightness-second-order}
	Let $\X^{\epsilon}$ be a sequence of rough paths, $\gamma \in ( \f 1 3, \f 1 2 - \f 1 {p})$, such that $\X(0)=0$, and $$\sup_{\epsilon \in (0,1]}\E \left(\Vert \X^{\epsilon}\Vert_{\gamma} \right)^{p} < \infty,$$
	then $\X^{\epsilon}$ is tight in $\FC^{\gamma'}$ for every $\f 1 3 <\gamma' < \gamma$.
\end{lemma}
\begin{proof}
	Choose $\alpha \in ( \gamma', \gamma)$, as $\rho_{\alpha}(\X,0) \leq \Vert \X \Vert_{\alpha} + \Vert \X \Vert_{\alpha}^2$ we obtain
	\begin{align*}
	\P  \left( \rho_{\alpha} (\X^{\epsilon}, 0) > R \right) &\leq \f {\E \left(  \rho_{\alpha} (\X^{\epsilon}, 0) \right)^{\f p 2}} {R^{\f p 2}}
	 \leq \f {\E \left(  \Vert \X \Vert_{\alpha} + \Vert \X \Vert_{\alpha}^2 \right)^{\f p 2}} {R^{\f p 2}}
	\lesssim \f C {R^{\f p 2}}.
	\end{align*}
	This proves the claim by Lemma \ref{compact}.
\end{proof}

\subsubsection{Interpreting the effective dynamics by classical equations}
\label{sec:explain}
We now explain what  the limiting equation means in the classical sense. Our set up is the following. 

\begin{assumption}
Let $X_t=( X^W_t, X^Z _t)$, where $X^W_t$ is a $n$-dimensional  possibly correlated Wiener process and $X^Z_t$ a $N-n$-dimensional Hermite process. The two components $X^W_t$ and $X^Z_t$ are independent, we  set
$$2A:=\left(\begin{matrix}\cov (X^W)&0\\0 &0
\end{matrix}\right). $$
We write  $A^{i,j}$ for the components of $A$.
We are concerned with the classical interpretation of the rough differential equation
$\dot x_t=F(x_t) d\X_t$,
where $F: \R^d \to \L(\R^N, \R^d)$ is a $\C_b^3$ map, $\X=(X, \XX+(t-s)A)$
and $\XX=(\XX^{i,j})$ is given by
$\XX^{i,j}_{0,t} =\int _0^t X_s^i \;dX_s^j$
interpreted as  It\^o integrals if $i,j \leq n$, otherwise as Young integrals.  
\end{assumption}

We  show that the rough differential equation (\ref{effective-rde}) is really the same as the equations given in part 3 of Theorem \ref{thm-homo-2}.  Without loss of generality, we will assume our solution is defined on the interval $[0,1]$.

According to Theorem 8.4 in \cite{Friz-Hairer}, see also \cite{Lyons94, Friz-Victoir},  there exists a unique solution to our rough differential equation
in the controlled rough path space $D_{X}^{2\alpha}([0,1]; \R^d)$,  where $\alpha>\f 13$.  The solution exists global in time 
and the full controlled process is given by $(x_s, F(x_s))$, which means $x_{s,t} = F(x_s) X_{s,t} + R^1_{s,t}$, where $ \Vert R^1 \Vert_{2 \alpha} < \infty$.
By Lemma 7.3 in \cite{Friz-Hairer} given a controlled rough path $(Y,Y') \in D_{X}^{2\alpha}([0,1]; \R^d)$ and a function $\phi \in \C^2_b$, then  $(\phi(Y),\phi(Y)') $ is also a controlled rough path in $ D_{X}^{2\alpha}([0,1]; \R^d)$, where $\phi(Y)' = D\phi(Y) Y'$. 
In our case $F \in \C^3_b$, thus,
$ (F(x_s),DF(x_s)F(x_s)) \in D_{X}^{2\alpha}([0,1]; \R^d)$ and
$$\begin{aligned} x_t -x_0 &=
 \int_0^t F(x_s) d\X_s\\
 &= x_0 + \lim_{|\CP|\to 0} \sum_{[u,v]\in \CP} F(x_u) X_{u,v} + DF(x_u) F(x_u) \XX_{u,v} +  DF(x_u) F(x_u) (v-u) A.
\end{aligned}$$
In components these are just, for  $l=1, \dots, d$, 
\begin{align*}
x^l_t =& x^l_0 + \lim_{|\CP|\to 0}\sum_{[u,v]\in \CP}  \sum_{k=1}^N F^l_k(x_u) X^k_{u,v} \\
&+ \sum_{l'=1}^d \sum_{i,j=1}^N DF(x_u)^{l,l',i} F^{l'}_j (x_u) \XX^{i,j}_{u,v} +  DF(x_u)^{l,l',i} F^{l'}_j (x_u) (v-u) A^{i,j}
\end{align*}
By assumption \ref{assumption-multi-scale} (2) the terms containing $\XX^{i,j}$, where $i \vee j >n $, do not contribute to the limit, hence we may neglect them, see also Lemma 4.2 in \cite{Friz-Hairer}.
We will drop these terms and use  $A^{i,j} = 0$  with only $i \vee j >n$. Let
\begin{align*}
I_1(\CP)=& \sum_{[u,v]\in \CP}  \sum_{k=1}^n F^l_k(x_u) X^k_{u,v} 
\\
&+ \sum_{l'=1}^d \sum_{i,j=1}^n DF(x_u)^{l,l',i} F^{l'}_j (x_u)\XX^{i,j}_{u,v} +  DF(x_u)^{l,l',i} F^{l'}_j (x_u) (v-u) A^{i,j}.\\
I_2(\CP)=& \sum_{[u,v]\in \CP}  \sum_{k=n+1}^N F^l_k(x_u) X^k_{u,v} 
\end{align*}
Now $I_2(\CP)$ gives rise the classical Young integrals $\int F_k^l(x_r) dX_r$. 
For $I_1$ we write $X^W=(X^1, \dots, X^n)$ as a linear combination of a standard $n$ dimensional Wiener $W$. Let $U$ be given such that  $U^TU=2A$ so $X^W = UW$. Then $ \XX^{i,j}_{u,v} = 2A^{i,j}_{u,v} \WW^{i,j}_{u,v}$, where $\WW^{i,j}$ denotes the It\^o lift of $W$, ( $\WW^{i,j}_{u,v} = \int_u^v X^i_{u,r} dX^j_r$).   We obtain,
\begin{align*}
I_1(\CP) &= \sum_{[u,v]\in \CP}  \sum_{k=1}^n F^l_k(x_u) \sum_{q=1} U^{k,q} W^q_{u,v} \\
&+ \sum_{l'=1}^d \sum_{i,j=1}^n DF(x_u)^{l,l',i} F^{l'}_j(x_u) 2A^{i,j}_{u,v} \WW^{i,j}_{u,v} +  DF(x_u)^{l,l',i} F^{l'}_j (x_u) (v-u) A^{i,j}.
\end{align*} 
Now, by Proposition 3.5 and Theorem 9.1 in \cite{Friz-Hairer} $ \lim_{|\CP|\to 0} I_1(\CP)$ coincides almost surely with the proclaimed Stratonovich integrals as the term $DF(x_u)^{l,l',i} F^{l'}_j (x_u) (v-u) A^{i,j}$ corresponds exactly the Stratonovich correction. We may now conclude our explanation.

\bigskip

Finally, we conclude the paper with a question. \\

{\bf Open Problem.} For Theorem A and B to hold,  the only restriction on the Hermite rank of the functions $G_k$ comes from the lack of integral bound  (\ref{integrable-1}). We can only prove this bound when $H^*(m)\in [0,\f 12]$. Our question is:  Can one lift the restriction  $H^*(m)<0$, and still obtain the bound (\ref{integrable-1})?  A proposal for obtaining this is to depart from the H\"older path approach used here and take on the p-variation rough path formulation instead. In  \cite{Chevyrev-Friz-Korepanov-Melbourne-Zhang}, the authors have improve their regularity assumption from their previous work by using the  p-variation rough path formulation instead of the H\"older one. They were studying the diffusive homogenisation problem, for this  they managed to include $p=\f 16$.

\newcommand{\etalchar}[1]{$^{#1}$}

\end{document}